\documentclass[12pt]{amsart}
\usepackage{graphicx}
\usepackage{amsmath}
\usepackage{subfigure}
\usepackage{amssymb} 
\usepackage[usenames,dvipsnames]{color}
\usepackage{pinlabel}
\usepackage{mathtools}
\usepackage{float}
\usepackage{tikz}
  \usetikzlibrary{hobby}
  \usetikzlibrary{patterns}

\newtheorem{thm}{Theorem}[section]  
\newtheorem{cor}[thm]{Corollary}

\newtheorem{defin}[thm]{Definition} 
\newtheorem{lemma}[thm]{Lemma} 
 
\newtheorem{prop}[thm]{Proposition} 
\newtheorem{conj}[thm]{Conjecture} 
 
\newtheorem*{defin*}{Definition}

\newcommand{\bbb}{\mbox{$\beta$}}

\newcommand{\eee}{\mbox{$\epsilon$}} 

\newcommand{\Ggg}{\mbox{$\Gamma$}}

\newcommand{\Bb}{\mbox{$\mathcal B$}}

\newcommand{\Rrr}{\mbox{$\mathbb{R} $}}

\newcommand{\bdd}{\mbox{$\partial$}}

\def\Diff{\mathop{\rm Diff}\nolimits}
\newcommand{\fakecoprod}{\mathbin{\rotatebox[origin=c]{180}{$\Pi$}}}
\newcommand{\underln}[1]{\underline{\smash{#1}}}
\newcommand{\ppap}{T_{t, \theta}}

\newcommand{\ra}{\rightarrow}
\newcommand{\pr}{\prime}
\newcommand{\de}{\partial}
\newcommand{\te}{\theta}

\begin{document}  

\title{Powell moves and the Goeritz group}   

\author{Michael Freedman}
\address{\hskip-\parindent
        Microsoft Station Q\\ 
        University of California\\
        Santa Barbara, CA 93106-6105\\ 
        and\\
        Mathematics Department\\
        University of California\\
        Santa Barbara, CA 93106-3080 USA}
\email{michaelf@microsoft.com}

\author{Martin Scharlemann}
\address{\hskip-\parindent
        Martin Scharlemann\\
        Mathematics Department\\
        University of California\\
        Santa Barbara, CA 93106-3080 USA}
\email{mgscharl@math.ucsb.edu}


\date{\today}

\begin{abstract}  In 1980 J. Powell \cite{Po} proposed that five specific elements sufficed to generate the Goeritz group of any Heegaard splitting of $S^3$, extending work of Goeritz \cite{Go} on genus $2$ splittings.  Here we prove that Powell's conjecture was correct for splittings of genus $3$ as well, and discuss a framework for deciding the truth of the conjecture for higher genus splittings.
\end{abstract}

\maketitle

Following early work of Goeritz \cite{Go}, the {\em genus $g$ Goeritz group} of the $3$-sphere can be described as the isotopy classes of orientation-preserving homeomorphisms of the $3$-sphere that leave the standard genus $g$ Heegaard surface $T_g$ invariant.  Goeritz identified a finite set of generators
for the genus $2$ Goeritz group; that work has been recently updated, extended and completed, to give a full picture of the group (see \cite{Sc}, \cite{Ak}, \cite{Cho}).  
In 1980 J. Powell \cite{Po} extended Goeritz' set of generators to a set of five elements that he believed would generate the Goeritz group for any fixed higher-genus splitting.  
Unfortunately, his proof that these suffice contained a serious gap; here we prove that they do suffice for genus $3$ splittings, using largely techniques that were unknown in 1980, and introduce methods that could be helpful in deciding the conjecture for higher genus splittings.

Powell's actual viewpoint on the Goeritz group, which we will adopt, is framed somewhat differently.  Following Johnson-McCullough \cite{JM} (who extend the notion to arbitrary compact orientable manifolds) consider the space of left cosets $\Diff(S^3)/\Diff(S^3, T_g)$, where $\Diff(S^3, T_g)$ consists of those orientation-preserving diffeomorphisms of $S^3$ that carry $T_g$ to itself.  The fundamental group $P_g$ of this space projects to the genus $g$ Goeritz group (with kernel $\pi_1(SO(3)) = \mathbb{Z}_2$ \cite[p.197]{Po}) as follows:  A non-trivial element is represented by an isotopy of $T_g$ in $S^3$ that begins with the identity and ends with a diffeomorphism of $S^3$ that takes $T_g$ to itself; this diffeomorphism of the pair $(S^3, T_g)$ represents an element of the Goeritz group as defined earlier.  The advantage of this point of view\footnote{ The $Z_2$ distinction (unimportant for our purposes) 
is best illustrated by Powell's generator $D_\eta$ in Figure \ref{fig:PowellPic}.  It has order $g$ when viewed as a diffeomorphism $(S^3, T_g) \to (S^3, T_g)$ but is of order $2g$ when viewed as an isotopy of $S^3$ to itself.} is that an element of the Goeritz group can be viewed quite vividly: it is a sort of excursion of $T_g$ in $S^3$ that begins and ends with the standard picture of $T_g \subset S^3$.  These excursions are what are pictured by Powell in Figure \ref{fig:PowellPic}.

\begin{figure}[ht!]
    \centering
    \includegraphics[scale=1]{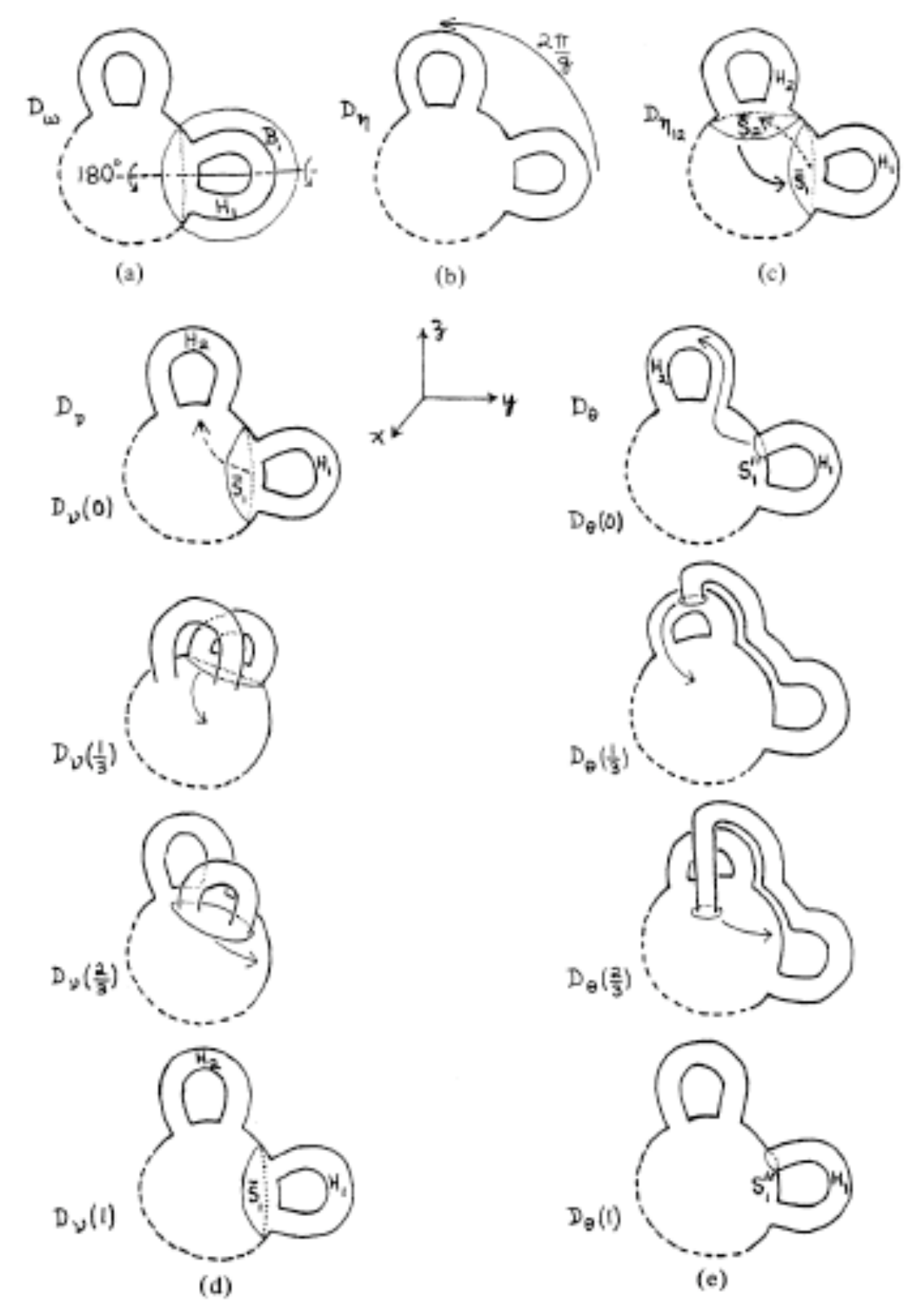}
    \caption{Powell's proposed generators (from \cite{Po})}
    \label{fig:PowellPic}
    \end{figure}

The study of motion groups of three dimensional objects has a deep connection to the emerging study of $3 + 1 - D$ topological quantum field theory (TQFT) and also to four dimensional questions such as the smooth Schoenflies problem. 
One may think of $P_g$ as a higher dimensional analog of the braid group $B_n$, the motion group of $n$-points in the plane. In the last 50 years $B_n$ has been closely studied with its unitary ($2 + 1 - D$ TQFT) representations becoming the main focus in the last 30 years. With increasing interest in $3 + 1 - D$ TQFTs, $P_g$ and its unitary representations are a natural target for condensed matter theorists.     

We would like to thank Alex Zupan for helpful conversations; his upcoming \cite{Zu} interleaves the Powell Conjecture with the question of the connectivity of the complexes of reducing spheres that arise from Heegaard splittings of $S^3$.

\section{Composing Powell generators} \label{sect:compose}

Let $T_g \subset S^3$ be the standard genus $g$ Heegaard surface in $S^3$, dividing $S^3$ into the genus $g$ handlebodies $A_g$ and $B_g$. 

\medskip

\begin{defin}  Any finite composition of Powell generators, illustrated in Figure \ref{fig:PowellPic}, will be called a {\em Powell move}.
\end{defin}
 
Suppose $T \subset S^3$ is a genus $g$ Heegaard surface, and $h, h': (S^3, T) \to (S^3, T_g)$ are two orientation-preserving homeomorphisms.

\begin{defin}  $h, h'$ are \em{Powell equivalent} if $h' h^{-1}: T_g \to T_g$ is isotopic in $T_g$ to a Powell move.
\end{defin}

\begin{conj}[Powell] Any $h, h'$ are Powell equivalent.
\end{conj}

The genus two case $g = 2$ is Goeritz's Theorem.
\bigskip

Powell's description of $T_g \subset S^3$ begins with a round $2$-sphere in $S^3$, to which is then connect-summed at each of $g$ points a standard unknotted torus.  These summands will be called the standard genus $1$ summands; their complement is a $g$-punctured sphere.  The two generators $D_{\eta}$ and $D_{\eta_{12}}$ show that any braid move of the standard genus $1$ summands around the complementary $g$-punctured sphere is a Powell move.   It follows that there will be no need to keep track of the labels of the standard summands.  Furthermore, $D_{\omega}$ allows us not to worry about the orientation of the longitude of the $1$-handle within each summand.  
\medskip

{\bf Example:}  Suppose $S^3 = A \cup_T \cup B$ is a genus $g$ Heegaard in $S^3$ and $\{a_1, ..., a_g\}$ and $\{b_1, ..., b_g\}$ are orthogonal sets of disks in $A$ and $B$ respectively (orthogonal means $|a_i \cap b_j| = \delta_{i,j})$.  There is an obvious orientation preserving homeomorphism of $S^3$ that carries $T$ to $T_g$  and each pair $\bdd a_i, \bdd b_i$ to the meridian and longitude respectively of one of the standard summands of $T_g$.  It follows from the comments above that any two such homeomorphisms are Powell equivalent.

More broadly, but in a similar spirit we have:
 \begin{lemma}   \label{lemma:braid2} 
 Any braid move of a collection of standard genus $1$ summands over their complementary surface is a Powell move.
  \end{lemma}  
   Figure \ref{fig:braid1} shows an example in the case of two summands.
   
   \begin{figure}[ht!]
    \centering
    \includegraphics[scale=0.75]{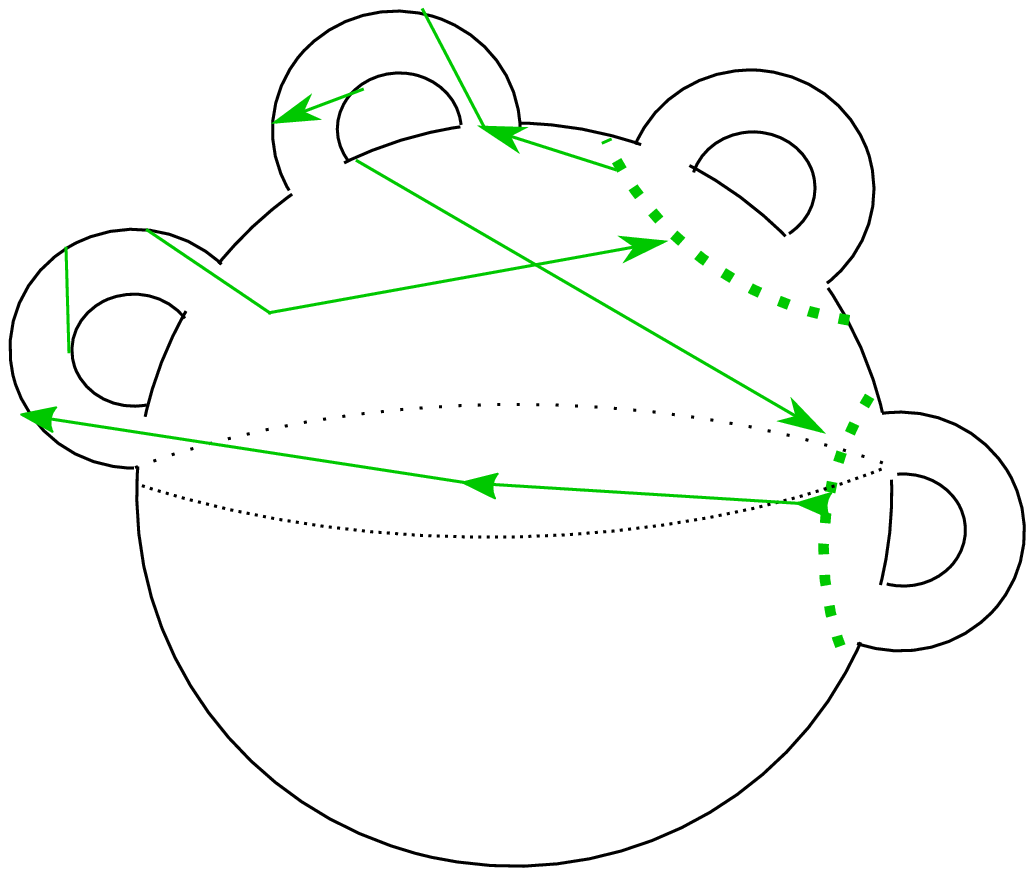}
    \label{fig:braid1}
    \end{figure}

 \begin{proof}  Let $\mathcal{S}$ be the collection of standard genus $1$ summands which we are isotoping.  From the brief discussion above it suffices to assume that the first Powell standard solid handle $H_1$ is in  $\mathcal{S}$ and $H_2$ is not, and then to show that any braid move of $H_1$ around the punctured torus $T \cap H_2$  is a Powell move.   The composition $D_{\theta}D_{\omega}D_{\theta}D_{\omega}$ isotopes $H_1$ around the longitude of $H_2$.  The generator $D_{\nu}$ isotopes $H_1$ around the meridian of $H_2$.  But any braid move of a point in a punctured torus is a composition of such circuits around the meridian and around the longitude.  
 \end{proof}

Let $\{c_1, ..., c_{g-1}\}$ be the disjoint separating circles on $T_g$ shown in Figure \ref{fig:circlesinT}, with each $c_i$ separating the first $i$ standard summands from the last $g-i$ standard summands.  Note that each $c_i$ bounds a disk in both $A$ and $B$ and so defines a reducing sphere $S_i$ for $T_g$.  

 \begin{figure}[ht!]
\labellist
\small\hair 2pt
\pinlabel  $c_1$ at 110 140
\pinlabel  $c_2$ at 167 133
\pinlabel  $c_{g-1}$ at 240 135
\endlabellist
    \centering
    \includegraphics[scale=0.75]{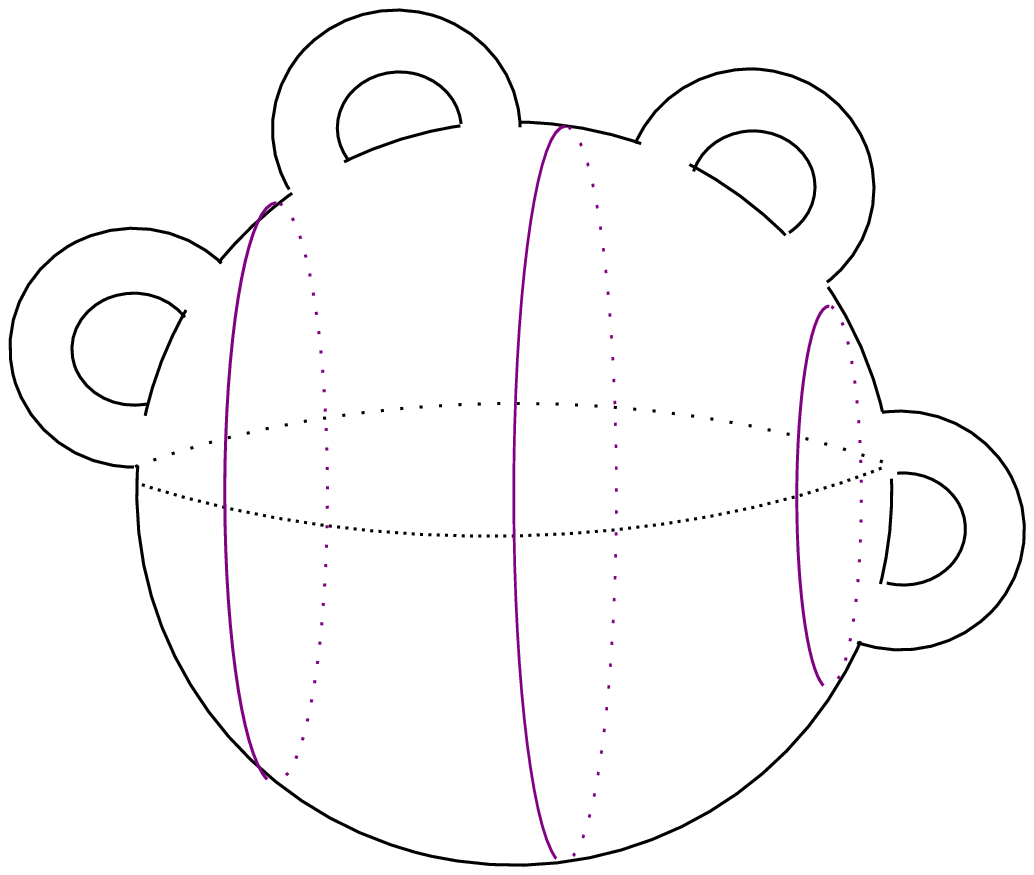}
     \caption{} \label{fig:circlesinT}
    \end{figure}

{\bf Extension 1:}  Suppose $S^3 = A \cup_T \cup B$ is a genus $g$ Heegaard surface in $S^3$ as above.  Let $g_1 + g_2 = g$.  Suppose  $\{a_1, ..., a_{g_1}\}$ and $\{b_1, ..., b_{g_1}\}$ are orthogonal sets of disks in $A$ and $B$ respectively.  
There are (many) orientation preserving homeomormorphisms of $S^3$ that take $T$ to $T_g$,  and each $\bdd a_i, \bdd b_i$ to the meridian and longitude respectively of one of the first $g_1$ standard summands of $T_g$. 

\begin{lemma} \label{lemma:orthog0} 
 {\em If the Powell Conjecture is true for genus $g_2$ splittings of $S^3$} then all such homeomorphisms $T\to T_g$ are Powell equivalent.

\end{lemma}

\begin{proof}  Let $h, h': (S^3, T) \to (S^3, T_g)$ be two such homeomorphisms.  Let $c = h'h^{-1}(c_{g_1}) \subset T_g$.  Then $c$, like $c_{g_1}$, separates $T_g$ into a genus $g_1$ surface containing the first $g_1$ handles and a genus $g_2$ surface disjoint from them. Also $c$, like $c_{g_1}$, bounds disks in both $A$ and $B$.
\medskip

If we cut off the first $g_1$ handles from $T_g$ then $c$ and $c_g$ both bound disks, and so can be isotoped to coincide, in the resulting genus $g_2$ surface. It follows then from Lemma \ref{lemma:braid2} that a Powell move will carry the first $g_1$ $1$-handles to themselves {\em and} take $c$ to $c_{g_1}$.  It follows that we may assume that $h'h^{-1}(c_{g_1})$ takes $c_{g_1}$ (and so the reducing sphere in which it lies) to itself.   The lemma then follows by operating separately on the genus $g_1$ and $g_2$ Heegaard splittings determined by $S_{g_1}$.
\end{proof}

\begin{defin} \label{defin:primitive}  Let $M = A \cup B$ be a Heegaard splitting of a compact manifold $M$.  A collection $\{a_1, ..., a_m\}$ of disjoint disks in $A$ is {\em primitive} if there is a collection of disjoint disks $\{b_1, ..., b_m\}$ in $B$ so that for every $i, j$, $|a_i \cap b_j| = \delta_{ij}$.

\end{defin}

{\bf Extension 2:} For $S^3 = A \cup_T \cup B$ as above, suppose $\{a_1, ..., a_g\}$ is a collection of $g$ primitive disks in $A$.  One can define a homeomorphism from $T$ to $T_g$  that takes each $\bdd a_i$ to the meridian of one of the standard summands of $T_g$.  

\begin{lemma} \label{lemma:orthog1}  
All such homeomorphisms are Powell equivalent.
\end{lemma}
\begin{proof}
Let $\{b_1, ..., b_g\}$ and $\{b'_1, ..., b'_g\}$ be different sets of  orthogonal disks in $B$.  For each $1 \leq i \leq g$, let $\bbb_i$ (resp  $\bbb'_i)$ denote the arc $\bdd b_i - \bdd a_i$ (resp $\bdd b'_i - \bdd a_i$).  

{\bf Special Case:}  $\{b_1, ..., b_g\}$ and $\{b'_1, ..., b'_g\}$ can be isotoped (leaving the $a_i$ invariant) so that the collection of curves $\{\bbb_1, ..., \bbb_g\}$ and $\{\bbb'_1, ..., \bbb'_g\}$ are disjoint.  

\medskip

Let $P$ be the planar surface obtained by compressing $T$ along $\{a_1, ..., a_g\}$.  For each $1 \leq i \leq g$ $\bbb_i \cup \bbb'_i$ forms a circle $c_i$ in $P$; if any $c_i$ bounds a disk in $P$  then the corresponding disks $b_i$ and $b'_i$ can be isotoped to coincide; the proof now proceeds by induction on the size of the set $\mathcal{S}$ of indices for which this is true; we may as well take $\mathcal{S} = \{1, ..., k\}$; when $k = g$ we are done.  Temporarily ignore the boundary components of $P$ that correspond to all $a_i, i \leq k$ and choose a circle $c_{\ell}, \ell > k$ that is then innermost. Then between $\bdd b_{\ell}$ and $\bdd b'_{\ell}$ there are only $1$-handle summands whose indices lie in $\mathcal{S}$.  Isotope these once around the curve $a_{\ell}$ (a sequence of moves each corresponding to Powell generator  $D_\nu$).  This isotopy (Powell equivalent to the identity) moves $b_{\ell}$ to $b'_{\ell}$, adding another index to $\mathcal{S}$ and completing the proof in this case.  

\medskip

The general case is proven via induction on the number of components of $b_i \cap b'_j, i, j \in \{1, ..., g\}$. We can assume that all components of intersection are arcs and say that such an arc is outermost in $b'_i$, say, if it cuts off a subdisk of $b'_i$ that contains neither another arc of intersection nor the point $b'_i \cap a_i$.   Choose an arc of intersection $\gamma$ that is outermost in, say,  $b'_i$ and let $b_j$ be the other disk containing $\gamma$.  Replace the subdisk of $b_j$ that is cut off by $\gamma$ (on the side disjoint from the point $b_j \cap a_j$) by the outermost subdisk of $b'_i$ cut off by $\gamma$.  This converts $\{b_1, ..., b_g\}$ to a disjoint family (so an equivalent family, per the Special Case) that intersects $\{b'_1, ..., b'_g\}$ in fewer arcs, as required.  
\end{proof}

\medskip  
It will later be shown (see Corollary \ref{cor:orthog2}) that the choice of just a single primitive disk determines the Powell equivalence type.  

\bigskip

{\bf Cautionary Note:}   Suppose $\{a_1, ..., a_{g_1}\}$ is a non-separating collection of disks in $A$ and $c$ is a simple closed curve $c$ in $T$ that separates $T$ into one component containing the curves  $\{\bdd a_1, ..., \bdd a_{g_1}\}$ and the other a genus $g_2$ surface. Then $c$ automatically bounds a disk in $A$; if $c$ also bounds a disk in $B$ then there is an orientation preserving homeomorphism $h:(S^3, T, c) \to (S^3, T_g, c_{g_1})$,.  The disks bounded by $c$ in $A$ and $B$ divide each into handlebodies, so the two pieces into which $T$ is divided are each Heegaard splittings of the $3$-balls in which they lie.  Moreover, if the Powell Conjecture is true in genus $g_1$ and genus $g - g_1$, then any two such homeomorphisms are Powell equivalent. But there is no obvious reason why different choices of $c$ will give Powell equivalent homeomorphisms, since not all braid moves on the $a_i$ (even if the $a_i$ are a primitive set) are Powell moves.  For example, if we rotate $a_1$ around a $\bdd$ reducing disk for $B$ whose boundary lies on the genus $g_2$ side of $c$, $c$ will be taken to a curve $c'$ that satisfies the same conditions.  Yet it is not immediately apparent that such an isotopy of $T$ to itself is a Powell move.  
\bigskip

\section{Powell-like moves in a more general setting} \label{sect:extend}

Following Lemma \ref{lemma:braid2} and the remarks that precede it we can focus on four types of Goeritz elements, each of which is a Powell move, but less stringently defined than Powell's originals:  

\begin{enumerate}
\item Let $S \subset S^3$ be a reducing sphere for $T_g$ bounding a ball $B$ containing $g_1$ of the standard genus $1$ summands.  A {\em standard bubble move} is an isotopy of $B$ through some path in $T_g - B$ that returns $B, B \cap T$ to itself, see Figure \ref{fig:bubble}.

 \begin{figure}[ht!]
    \centering
    \includegraphics[scale=0.5]{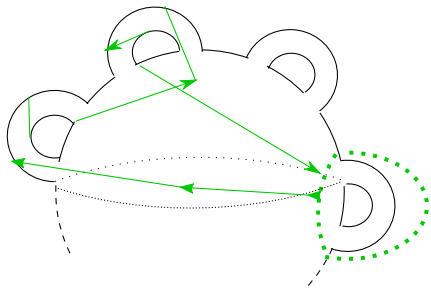}
    \caption{Standard bubble move} \label{fig:bubble}
    \end{figure}

\item Let $B$ be a bubble containing just a single standard genus $1$ summand.  A {\em standard flip} is the homeomorphism $(B, B \cap T) \to (B, B \cap T)$ shown in Figure \ref{fig:PowellPic} as Powell's move $D_\omega$.

 \begin{figure}[ht!]
    \centering
    \includegraphics[scale=0.5]{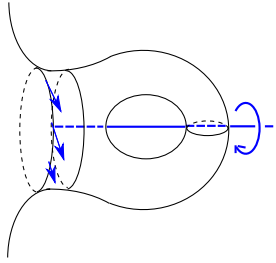}
    \caption{Standard flip}  \label{fig:flip}
    \end{figure}

\item Let $B_1, B_2$ be disjoint bubbles, each containing just a standard genus $1$ summand, and let $v \subset T_g - (B_1 \cup B_2)$ be an arc connecting them. Let $B$ be the reducing ball obtained by attaching the $1$-handle regular neighborhood of $v$ to $B_1, B_2$.  A {\em standard switch} is the homeomorphism $(B, B \cap T_g) \to (B, B \cap T_g)$ shown in Figure \ref{fig:PowellPic} as Powell's move $D_{\eta_{12}}$. 
\medskip

 \begin{figure}[ht!]
    \centering
    \includegraphics[scale=0.5]{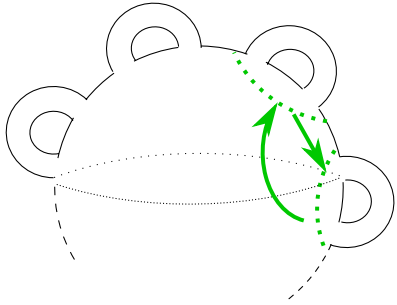}
    \caption{Standard switch}  \label{fig:switch}
    \end{figure}

\item Again let $B_1, B_2$ be disjoint bubbles, each containing just a standard genus $1$ summand.  For $i = 1, 2$ let $\mu_i \subset B$ be a meridian disk for $T_g \cap B_i$ and $\ell_i \subset A$ be a longitudinal disk for $T_g \cap B_i$.  Let $v \subset T_g$ be an embedded arc connecting $\bdd \mu_1$ to $\bdd \ell_2$ with the interior of $v$ disjoint from $\mu_1, \mu_2, \ell_1, \ell_2$.  Then $\mu_1 \cup \ell_2 \cup v$ is called a {\em standard eyeglass} in $T_g$  see Figure \ref{fig:eyeglass}.  The disks $\mu$ and $\ell$ are called the {\em lenses} of the eyeglass, $v$ is the {\em bridge}.  The eyeglass defines is a natural automorphism $(S^3, T_g) \to (S^3, T_g)$ called an {\em eyeglass twist} which is supported on a $3$-ball regular neighborhood of the eyeglass, see Figure \ref{fig:eyeglass1}.  Powell's generator $D_\theta$ is a standard eyeglass twist, in fact the model for this discussion.

\begin{figure}[ht!]
    \centering
    \includegraphics[scale=0.7]{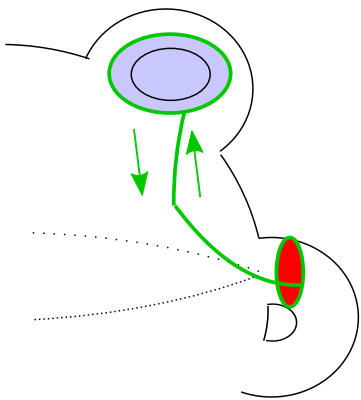}
    \caption{Standard eyeglass}  \label{fig:eyeglass}
    \end{figure}

 \begin{figure}[ht!]
    \centering
    \includegraphics[scale=0.75]{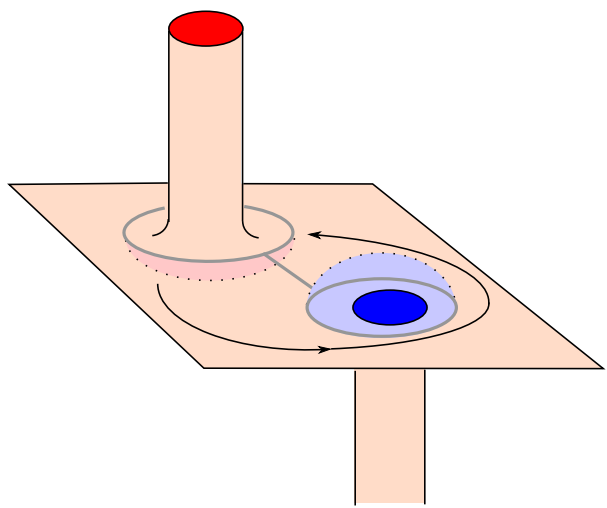}
    \caption{Eyeglass twist}  \label{fig:eyeglass1}
    \end{figure}

\end{enumerate}

Each of these moves can be put in a more general setting.  Suppose $M$ is any compact oriented $3$-manifold and $T$ is a stabilized Heegaard splitting, with some genus $1$ summands of $T$ designated as standard.  Then all the definitions above make sense.  More generally, if we drop the word 'standard' all the concepts make sense even when specific genus $1$ summands have not been designated as standard.  In fact, the notion of eyeglass, and so eyeglass twist, can often apply even when the splitting $M = A \cup_T B$ is not stabilized.  

\begin{defin} \label{defin:eyeglass}
An {\em eyeglass} is the union of two disks, $\ell_a, \ell_b$ ( the {\em lenses} ) with an arc $v$ (the {\em bridge}) connecting their boundaries.  Suppose an eyeglass $\eta$ is embedded in $M$ so that the $1$-skeleton of $\eta$ (called the {\em frame}) lies in $T$, one lens is properly embedded in $A$, and the other lens is properly embedded in $B$.  The embedded $\eta$ defines a natural automorphism $(M, T) \to (M, T)$,  as illustrated in Figure \ref{fig:eyeglass1}, called an {\em eyeglass twist}.
\end{defin}

{\bf Remark:}  It will be useful later to note that a $\bdd$-compression of one of the lenses of an eyeglass $\eta$ into $T$ breaks the eyeglass twist around $\eta$ into a composition of two eyeglass twists $\eta_1 \eta_2$, where in each $\eta_i$ the $\bdd$-compressed disk is replaced by one of the disks that is the result of the $\bdd$-compression.  See Figure \ref{fig:eyeglass3}.  

 \begin{figure}[ht!]
    \centering
    \includegraphics[scale=0.55]{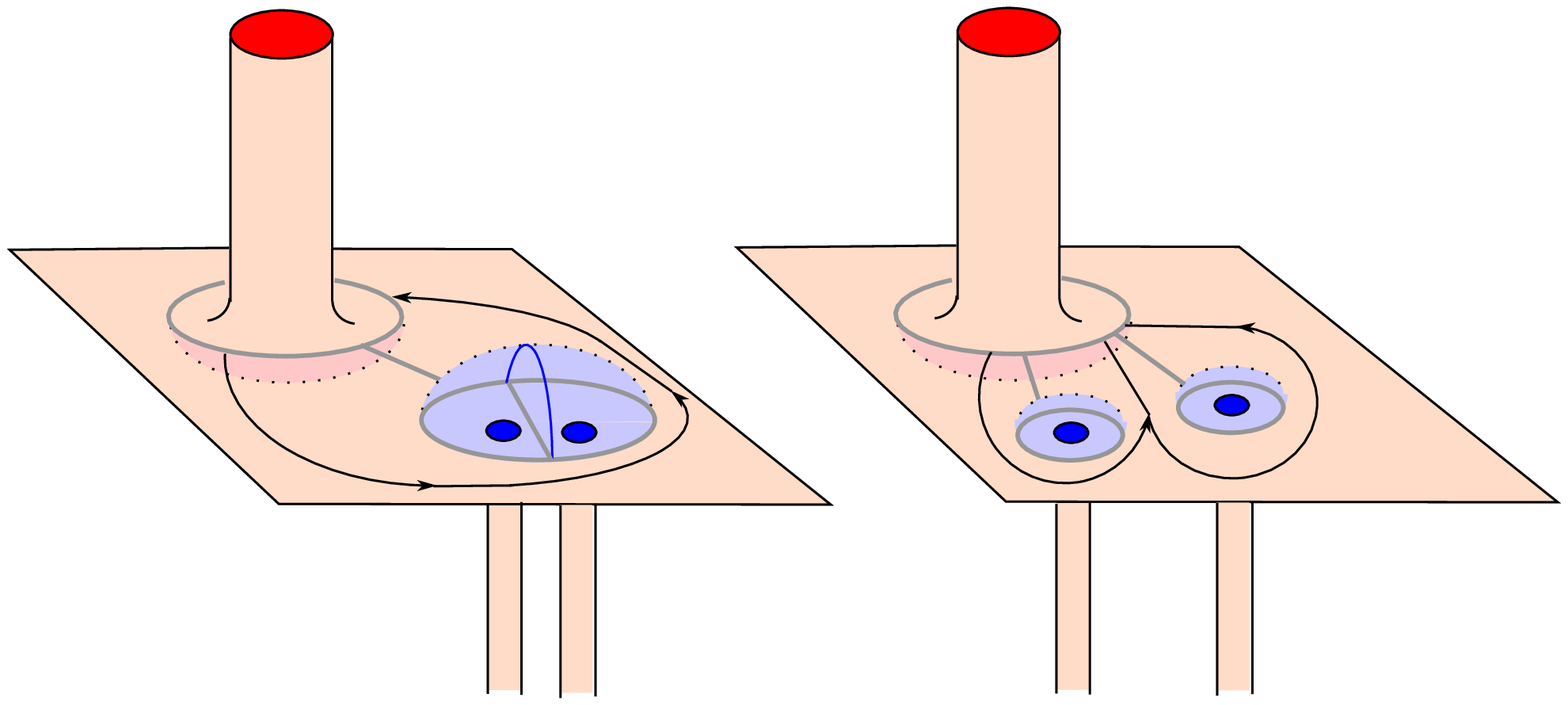}
   \caption{} \label{fig:eyeglass3}
    \end{figure}

\section{Weak reduction} \label{sect:weak}

Recall (\cite{CG}) that a Heegaard splitting $M = A \cup_T B$ is {\em weakly reducible} if there is an essential disk in $A$ and one in $B$ so that their boundaries are disjoint in $T$.  Since the two lenses in an eyeglass are disjoint, it follows that only weakly reducible splittings can support non-trivial eyeglass twists.  In this section we consider families of weakly reducing disks in Heegaard splittings of $S^3$ and describe conditions which guarantee that an eyeglass twist is a Powell move.  Of course the Powell Conjecture would assert that any eyeglass twist is a Powell move.  

\begin{defin} Suppose a non-empty collection $\{a_1, ..., a_m\} $ of disks embedded in $A$ is disjoint from a non-empty collection $\{b_1, ..., b_n\}$ of disks in $B$.  Then the pair of collections is called {\em weakly reducing}.  If the complement in $T$ of their boundaries $\{\bdd a_1, ..., \bdd a_m\} \cup \{\bdd b_1, ..., \bdd b_n\}$ consists of planar surfaces in $T$ then the pair of disk collections is {\em complete}.  If the complement is a single surface then they are called {\em non-separating}.  
\end{defin}

Suppose $\{a_1, ..., a_{g_1}\}$,  $\{b_{g_1+1}, ..., b_g\}$ is a complete non-separating weakly-reducing pair of disk collections for a genus $g$ splitting  $S^3 = A \cup_T B$.  Let  $c \subset T$ be a simple closed curve that is disjoint from  the boundaries of the two sets and separates one set from the other.  Then $c$ bounds a disk $a_c$ in $A$ and a disk  $b_c$ in $B$.  The union $a_c \cup b_c$ is a sphere, dividing $S^3$ into two balls, and $T$ intersects each ball in a Heegaard surface.   Waldhausen's theorem applied to each then shows that there is an orientation preserving homeomorphism $(S^3, T, c) \to (S^3, T_g, c_{g_1})$.  Moreover, {\em if the Powell Conjecture is true for genus $g_1$ and for genus $g - g_1$ splittings of $S^3$} then all such homeomorphisms are Powell equivalent. 
%
%
%

%
%
%

Consider the standard genus $g$ Heegaard splitting $S^3 = A \cup_{T_g} B$, in which we denote the meridians of the standard summands by $\{a_1, ..., a_{g}\}$ and the longitudes by $\{b_{1}, ..., b_g\}$.  In particular $|a_i \cap b_j| =  |\bdd a_i \cap \bdd b_j| = 1, 1 \leq i, j \leq g$.  The curve $c_{g_1} \subset T_g$ separates $T_g$ into two components, $T_A$ containing $\{\bdd a_1, ..., \bdd a_{g_1}\}$ and $T_B$ containing $\{\bdd b_{g_1+1}, ..., \bdd b_g\}$.  Let $P_A$ (resp $P_B$) be the planar surface $T_A - \{a_1, ..., a_{g_1}\}$ (resp $T_B - \{b_{g_1 + 1}, ..., b_g\}$) and let $P$ be the combined planar surface $P_A \cup_c P_b$.  

\begin{lemma} \label{lemma:eyeglass1o} Suppose $\eta$ is an eyeglass in $T_g$ whose lenses consists of some $a_i, i \leq g_1$ and some $b_j, j \geq g_1 + 1$ and whose bridge $v$ intersects $c_{g_1}$ exactly once.  Then an eyeglass twist along $\eta$ is a Powell move. 
\end{lemma} 

\begin{proof}
Observe that, after perhaps some conjugating Powell moves as described at the start of Section \ref{sect:compose}, Powell's generator $D_{\theta}$ can be viewed as an eyeglass twist around the same lenses and a bridge $v' \subset P$ which also intersects $c_{g_1}$ once.   
Now use Lemma \ref{lemma:braid2}  on the $1$ handles in $T_B$ to slide them around $T_A$ so that afterwards $v \cap T_A$ coincides with $v' \cap T_A \subset P_A$.  Then symmetrically slide the $1$-handles of $T_A$ around $T_B$ until  the entire $v = v' \subset P$.  
\end{proof}

\begin{lemma}  \label{lemma:eyeglass2o} Suppose $\eta$ is an eyeglass in $T_g$ whose lenses consist of a disk $a \subset A$ with $\bdd a \subset P_A$ and a disk $b \subset B$ with $\bdd b \subset P_B$.  Suppose further that the bridge $v$ intersects $c_{g_1}$ exactly once.  Then an eyeglass twist along $\eta$ is a Powell move. 
\end{lemma} 
\begin{proof}  As in the proof of Lemma \ref{lemma:eyeglass1o} we can assume that the bridge $v$ lies in $P = P_A \cup P_B$.  Then $\bdd a$ is coplanar with a number of boundary components of $P_A$ (informally, $\bdd a$ bounds a disk containing copies of some of the $a_i$).  Noting the remark following Definition \ref{defin:eyeglass}, we can view the twist around $\eta$ as a composition of twists around each of these components of $P_A$.  Do the same for $\bdd b$.  The result follows now from Lemma \ref{lemma:eyeglass1o}.
\end{proof}
\medskip

\begin{lemma}  \label{lemma:eyeglass3o} Suppose $\eta$ is an eyeglass in $T_g$ whose lenses consist of a disk $a \subset A$ with $\bdd a \subset T_A$ and a disk $b \subset B$ with $\bdd b \subset T_B$.  Suppose further that the bridge $v$ intersects $c_{g_1}$ exactly once.  Then an eyeglass twist along $\eta$ is a Powell move. 
\end{lemma} 

\begin{proof}  As before, we may assume that $v$ lies in $P$.  The proof is by induction on the number of arc components of $a \cap (a_1\cup ...\cup a_{g_1})$ and $b \cap (b_{g_1 + 1} \cup ... \cup b_g)$ .  If there are none, the result follows from Lemma \ref{lemma:eyeglass2o}.  Otherwise choose an arc of intersection that is outermost in, say, $a_i$.  Use the arc to $\bdd$-compress $a$, breaking up the twist around $\eta$ into the twist around two eyeglass curves as illustrated in Figure \ref{fig:eyeglass3}.  (If one of the new $A$ lenses is inessential then $v$ could have been extended past the $\bdd$-compressing disk, replacing the arc of intersection by a new point of intersection of $a_i$ with $v$.  This can be removed as usual.)  The result now follows by induction.  (But, cautionary note, one of the new eyeglasses has a bridge that intersects $a_i$.  Fortunately, having $v$ disjoint from $\{a_1, ..., a_{g_1}\}$ is not part of our hypothesis but was achieved by argument.  This condition is needed, else the boundary compression used in the argumentwould intersect $v$.)
\end{proof}

Here is an application: 

\begin{prop} \label{prop:orthog2o} Suppose  $b \subset B$ is a disk orthogonal to $a_1$.  Then there is a Powell move that leaves $a_1$ unchanged and carries $b_1$ to $b$.
\end{prop}

\begin{proof}  We can assume $\bdd b$ intersects a bicollar neighborhood $Y \subset T_g$ of $\bdd a_1$ in an arc parallel to but disjoint from $\bdd b_1 \cap Y$.  

 {\bf Special Case:}  $b$ and $b_1$ are disjoint.
 \medskip
 
Band $\bdd b$ to $\bdd b_1$ together using one of the two bands they cut off from the bicollar $Y$, creating a new disk $b_+$.  Push the interior of this band into $B$ to properly embed $b_+$ in $B$.  Consider $a_1$ and $b_+$ as lenses of an eyeglass $\eta$ whose bridge $v$ is one of the small arcs $(\bdd b \cap Y) - \bdd a_1$. It is easy to see (Figure \ref{fig:orthog2a}) that an appropriate eyeglass twist of $\eta$ will move $b_1$ to $b$.  The eyeglass visibly satisfies the criterion required by Lemma \ref{lemma:eyeglass3o} :  Let 
$c$ be the boundary of a regular neighborhood of $\bdd a_1 \cup b_1$.  Then $c$ bounds a punctured torus in $T_g$ (in fact the first standard summand) that  contains $a_1$, is disjoint from $b_+$, and intersects $v$ in a single point.  See Figure \ref{fig:orthog2b}. 

 \begin{figure}[ht!]
\labellist
\small\hair 2pt
\pinlabel  $\bdd b_1$ at 235 70
\pinlabel  $\bdd b$ at 300 30
\pinlabel  $\bdd b_+$ at 250 10
\pinlabel  $\bdd a_1$ at 110 65
\pinlabel  $v$ at 138 25
\endlabellist
    \centering
    \includegraphics[scale=0.75]{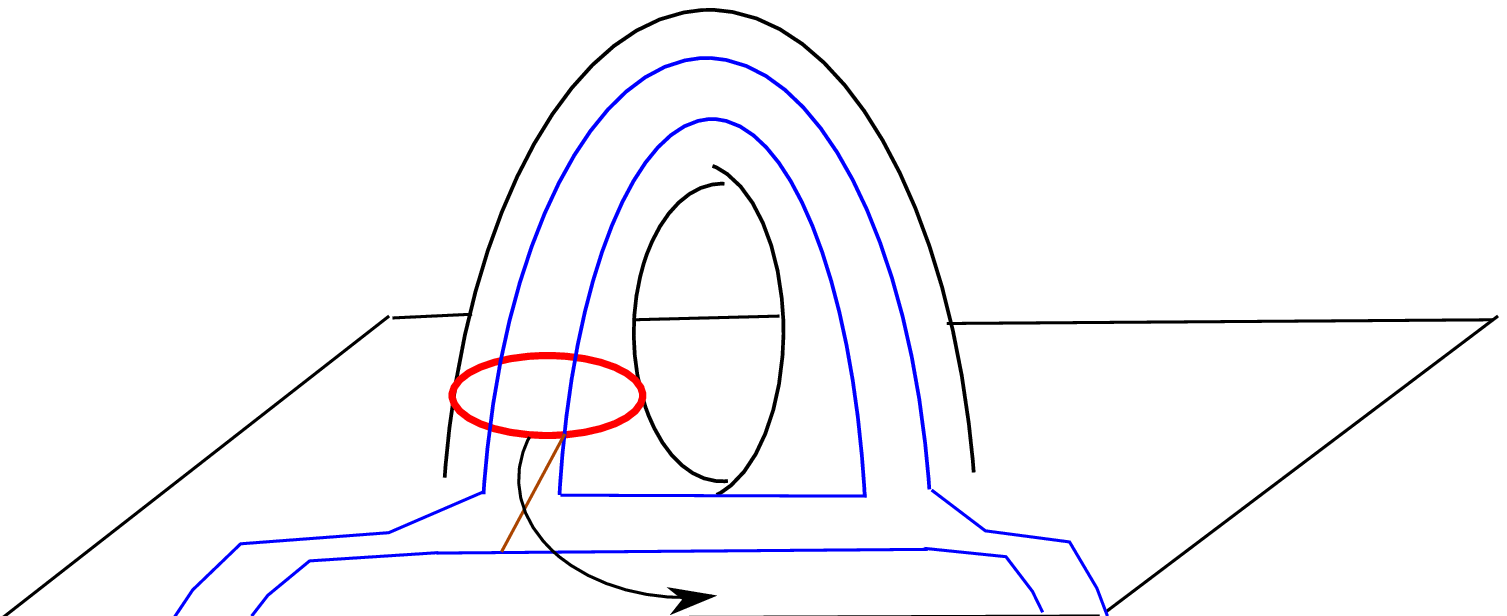}
\caption{} 
 \label{fig:orthog2a}
    \end{figure}
    
     \begin{figure}[ht!]
\labellist
\small\hair 2pt
\pinlabel  $c$ at 155 125
\pinlabel  $\bdd b_+$ at 250 10
\pinlabel  $\bdd a_1$ at 110 65
\pinlabel  $v$ at 135 25
\endlabellist
    \centering
    \includegraphics[scale=0.75]{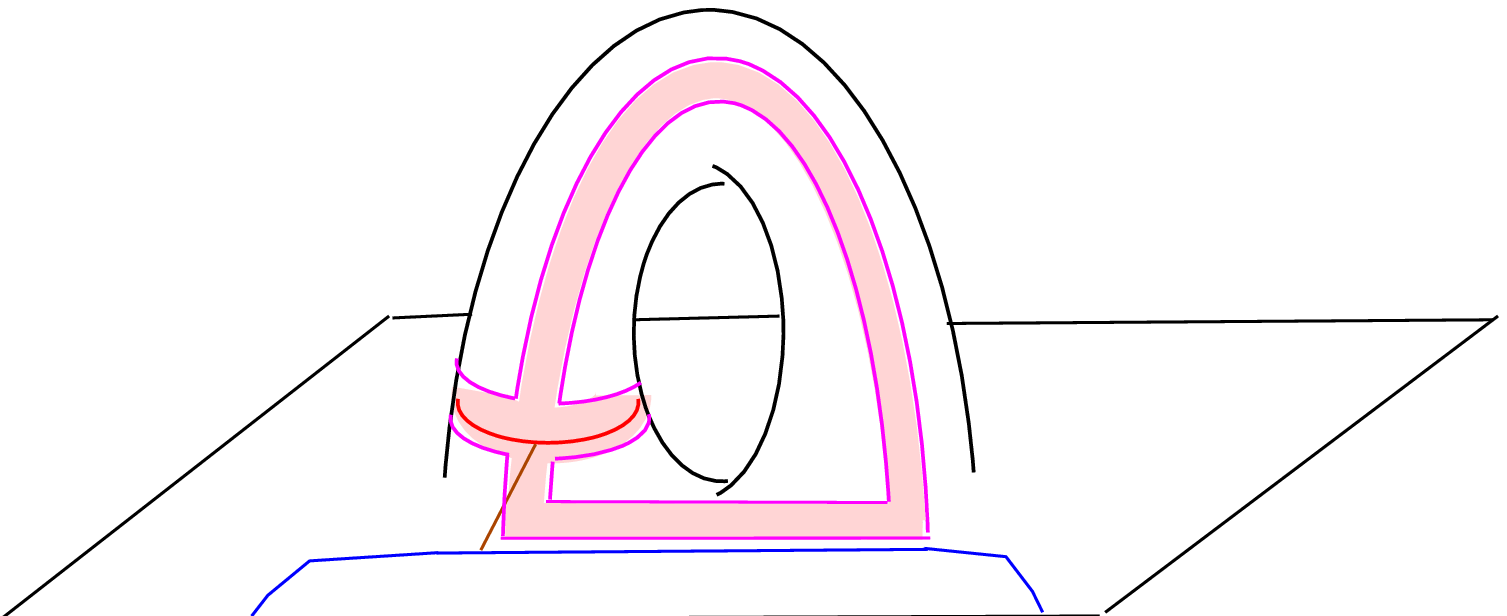}
\caption{} 
 \label{fig:orthog2b}
    \end{figure}

 The general case now follows much as in the proof of Lemma \ref{lemma:orthog1}, via induction on the number of components of $b \cap b_1$:  Suppose $b$ is a disk intersecting $b_1$ in $n$ arcs and the proposition is true for disks that intersect $b_1$ in fewer than $n$ arcs.  Let $\gamma$ be an arc of $b \cap b_1$ that is outermost in $b$ (with reference to the point $b \cap a_1$) and let $b_o \subset b$ be the disk $\gamma$ cuts off from $b$.  Replace the disk in $b_1$ cut off by $\gamma$ with $b_o$ to get a disk $b'$ orthogonal to $a_1$ that is disjoint from $b_1$ and intersects $b$ in fewer than $n$ arcs.  By the special case just proven there is a Powell move $\rho$ that carries $b'$ to $b_1$ and so carries $b$ to a disk $b''$ intersecting $\rho(b') = b_1$ in fewer than $n$ arcs.  By inductive assumption, there is a Powell move $\rho'$ carrying $b"$ to $b_1$.  Then $\rho' \rho(b) = \rho'(b'') = b_1$ as required. 
\end{proof}

\begin{cor}  \label{cor:orthog2}  Suppose the Powell conjecture is true for genus $g-1$ splittings of $S^3$ and $T \subset S^3$ is a genus $g$ splitting.  Then the choice of a single primitive disk $a \subset A$ (or primitive $b \subset B$) defines a Powell equivalence class of homeomorphism $(S^3, T) \to (S^3, T_g)$.
\end{cor}

\begin{proof}  Since $a$ is primitive, there is a disk $b \subset B$ whose boundary intersects the boundary of $a$ in a single point.  Choose a homeomorphism $h: T \to T_g$ which carries the pair $(a, b)$ to the pair $(a_1, b_1)$.  By Proposition \ref{prop:orthog2o} the Powell equivalence class of $h$ does not depend on $b$.  
The proof then follows by applying the inductive assumption to the genus $g-1$ side of the reducing sphere $S_{1}$ for $T_g$.
\end{proof}

Returning now to the general discussion, we drop the assumption that the $a_i$ and $b_j$ are primitive, but otherwise maintain the notation above.  Suppose $\{a_1, ..., a_{g_1}\}$ and $\{b_{g_1+1}, ..., b_g\}$ are a (not necessarily primitive) pair of non-separating weakly-reducing disk collections for $T \subset S^3$.  Let  $c \subset T$ be a simple closed curve that separates the two sets and $\eta$ is an eyeglass in $T$ whose lenses consists of a disk $a \subset A$ with $\bdd a \subset T_A$ and a disk $b \subset B$ with $\bdd b \subset T_B$.   

\begin{lemma} \label{lemma:eyeglass4}   Suppose that the bridge $v$ for $\eta$ intersects $c$ exactly once.   Then an eyeglass twist along $\eta$ does not change the Powell equivalence class of any homeomorphism $h:(S^3, T, c) \to (S^3, T_g, c_{g_1})$.
\end{lemma}

\begin{proof}  First note that there is such a homeomorphism:  the hypothesis guarantees that $c$ bounds a disk in both $A$ and $B$ and so is part of a reducing sphere for $T$.  Then Waldhausen's theorem applied to both $T_A$ and $T_B$ provides a homeomorphism $h$. (This homeomorphism is unique up to Powell equivalence if the Powell Conjecture is true for lower genus splittings.  We do not need that inductive assumption here.)  Now apply Lemma \ref{lemma:eyeglass3o}.   \end{proof}

\begin{thm} \label{thm:cindep}  Suppose $\{a_1, ..., a_{g_1}\}$ and $\{b_{g_1+1}, ..., b_g\}$ are a (not necessarily primitive) pair of non-separating weakly-reducing disk collections for $T \subset S^3$.  Then
\begin{itemize}
\item  there is a homeomorphism $h: T \to T_g$ so that $\{h(a_1), ..., h(a_{g_1})\}$ lie on the side of the separating sphere $S_{g_1}$ that contains $c_1$ and $\{h(b_{g_1+1}), ..., h(b_g)\}$ lie on the other side and
\item  If the Powell Conjecture is true for genus $g_1$ and genus $g - g_1$ splittings, any two such homeomorphisms are Powell equivalent.  
\end{itemize}
\end{thm}

\begin{proof}  Choose any simple closed curve $c \subset T$ that lies between the collections $\{\bdd a_1, ..., \bdd a_{g_1}\}$ and $\{\bdd b_{g_1+1}, ..., \bdd b_g\}$.  As noted in the proof of Lemma \ref{lemma:eyeglass4} there is a homeomorphism $h:(S^3, T, c) \to (S^3, T_g, c_{g_1})$ as required, and its Powell equivalence class depends only on the choice of $c$.  The goal then is to show that the Powell equivalence class does not even depend on $c$.  

As before, let $P$ be the connected planar surface $T - \{a_1, ..., a_{g_1} \cup b_{g_1+1}, ..., b_g\}$, so that $c \subset P$ separates the $2g_1$ components of $\bdd P$ corresponding to $\{\bdd a_1, ..., \bdd a_{g_1}\}$ from the $2(g - g_1)$ components corresponding to $\{\bdd b_{g_1+1}, ...,\bdd b_g\}$.   Call the former component $P_A$ and the latter $P_B$.  

Suppose $c' \subset P$ is another such simple closed curve.  Picturing $P$ as a $2g$ punctured sphere, a braid automorphism will move $c'$ to $c$.  So we need only show that any braid automorphism $\rho$ that moves the boundary components $\bdd P \cap P_A$ back to themselves (and so the boundary components of $\bdd P \cap P_B$ back to themselves) does not change the Powell equivalence class of $h$.  


 It is a classic result that the ``mixed'' braid group $\Bb_{g_1, g_2}$ on the $g$-punctured sphere $P_A \cup P_B$ can be generated by a set of $g_1 - 1$ half-twists in $P_A$ (call this subgroup $\Bb_a$), $g_2 - 1$, half-twists in $P_B$ (call this subgroup $\Bb_b$) and a single full twist $\sigma$ along a chosen arc $\gamma$ connecting a specific component $a_0$ of $\bdd P_A$ to a specific component $b_0$ of $\bdd P_B$.  This is an eyeglass twist with lenses $a_0$ and $b_0$ and bridge $\gamma$.  Choose $\gamma$ to be an arc crossing $c$ once.

Clearly the subgroups $\Bb_a$ and $\Bb_b$ commute.   The proof that $\rho \in \Bb_{g_1, g_2}$ is Powell equivalent to one that lies in $\Bb_a \times \Bb_b$ proceeds by induction on  the number $n_\sigma$, of occurrences of $\sigma$ in $\rho$ when expressed as a product of these generators.  If $\sigma$ does not appear, then $c$ is preserved by the braid.  With no loss the initial segment of $\rho$ can be written $\alpha \beta \sigma$, where $\alpha \in \Bb_a$ and $\beta \in \Bb_b$.  That is, $\rho = \alpha \beta \sigma \omega$, where $\omega \in \Bb_{g_1, g_2}$ has one less occurence of $\sigma$.  Then $\rho = ( \alpha \beta \sigma \beta^{-1}\alpha^{-1}) \alpha \beta \omega$.  But it is easy to see that  $(\alpha \beta \sigma \beta^{-1}\alpha^{-1})$ is an eyeglass twist along an eyeglass whose bridge is $\alpha \beta(\gamma)$, an arc that still crosses $c$ once.  The proof then follows from Lemma \ref{lemma:eyeglass4}.
\end{proof}

To state the implication of Theorem  \ref{thm:cindep} a little less formally:

\begin{cor} \label{cor:weakreduce} If the Powell Conjecture is true for genus $g_1$ and genus $g - g_1$, then a pair $\{a_1, ..., a_{g_1}\}$, $\{b_{g_1+1}, ..., b_g\}$ of non-separating weakly-reducing disk collections for $T \subset S^3$ determines a Powell equivalence class of homeomorphisms $h: T \to T_g$.
\end{cor}

In fact, only one set of disks is needed, so long as we know it has at least one complementary set:  

\begin{thm}  Suppose the collection $ \{b_{g_1+1}, ..., b_g\}$ of compressing disks in $B$ can be extended to two possibly different non-separating complete collections $\{a_1, ..., a_{g_1}\}, \{b_{g_1+1}, ..., b_g\}$ and $\{a'_1, ..., a'_{g_1}\}, \{b_{g_1+1}, ..., b_g\}$ of weakly reducing disks for $T$.  
 If the Powell Conjecture is true for genus $g_1$ and genus $g - g_1$, then the Powell equivalance classes determined by each extension via Theorem \ref{thm:cindep} are the same.
\end{thm}

\begin{proof}  
Let $H \subset S^3$ be the genus $g_1$ handlebody obtained from $A$ by attaching $2$-handles along $ \{b_{g_1+1}, ..., b_g\}$.  

{\em Special Case:}  $\{a_1, ..., a_{g_1}\}$ and $\{a'_1, ..., a'_{g_1}\}$ are disjoint.  

We first claim that, in this case, we can proceed from one family to the other by a sequence of substitutions of a single disk at a time. This is obvious (indeed there is nothing to prove) if each $a'_j$ is parallel to one of the $a_i$, so we induct on the number of $a'_j$ that are not parallel to any $a_i$.   With no loss of generality, say $a'_1$ is not parallel to any $a_i$.  Since the set $\{a_1, ..., a_{g_1}\}$ is non-separating in $H$,  $H - \{a_1, ..., a_{g_1}\}$ is a ball containing the disjoint collection of properly embedded disks $\{a'_1, ..., a'_{g_1}\}$.  Each $a_i$ gives rise to a pair of disks (called {\em twins}) on the boundary of the ball. Since $H - \{a'_1, ..., a'_{g_1}\}$ is connected, there is at least one $a_i$, say $a_1$, which is parallel to no $a'_j$ in $H$ and whose twins in the boundary of the ball lie on opposite sides of $a'_1$.  Then replacing $a'_1$ by $a_1$ changes only a single disk in $\{a'_1, ..., a'_{g_1}\}, \{b_{g_1+1}, ..., b_g\}$ and provides a non-separating complete collection $\{a_1, a'_2,  ..., a'_{g_1}\}, \{b_{g_1+1}, ..., b_g\}$  with more disks parallel in $A$ to disks in $\{a_1, ..., a_{g_1}\}$, completing the inductive step.  So, following the claim, we may as well assume that $a_i = a'_i$ for all $2 \leq i \leq g_1$.   

Let $W$ be the solid torus obtained from $H$ by deleting each $a_i = a'_i, 2 \leq i \leq g_1$.  The remaining disks $a_1, a'_1$ are disjoint meridian disks of $W$, dividing it into two cylindrical components $W_{\pm}$, each topologically a ball.  Now consider the $g - g_1$ properly embedded arcs  $\{\beta_{g_1}, ..., \beta_g \}$ in $W$, each arc $\beta_j$ dual to the disk $b_j \subset B$.  If, say, $W_-$ contains none of the $\beta_j$, then all  the $\beta_j$ lie in $W_+$.  Then there is a curve $c \subset \bdd A \cap W_+$ that separates $\{\bdd a'_1, \bdd a_1, ..., \bdd a_{g_1}\}$ from $ \{\bdd b_{g_1+1}, ..., \bdd b_g\}$.   Then in $S^3$,  $c$ bounds disks $a_c \subset A$ and $b_c \subset B$, and so provides a reducing sphere $a_c \cup b_c$ for the splitting $T$ that determines, as shown in Theorem \ref{thm:cindep}, the same Powell equivalence class for the pair of weakly compressing collections  $\{a_1, ...,  a_{g_1}\}, \{b_{g_1+1}, ..., b_g\}$ as it does for  $\{a'_1, a_2 ..., a_{g_1}\}, \{b_{g_1+1}, ..., b_g\}$, completing the argument in this case.  (In effect,  $\{a_1, ..., a_{g_1}\}$ and $ \{a'_1, ..., a'_{g_1}\}$ are merely different choices of complete collections of $\bdd$-reducing disks for $A$ in the side of the sphere $a_c \cup b_c$ on which they lie).

The remaining possibility in this special case is that some of the $\beta_j$ lie in each of $W_{\pm}$.   There is a curve $c \subset T$ that  separates $\{ \bdd a_1, ..., \bdd a_{g_1}\}$ from $ \{\bdd b_{g_1+1}, ..., \bdd b_g\}$ and intersects $\bdd a'_1$ in two points. (So $c$ bounds a disk $a_c \subset A$ intersecting the disk $a'_1$ in a single arc $\gamma$.)    Via Theorem \ref{thm:cindep} this determines a Powell equivalence class for the pair of weakly compressing collections  $\{a_1, ...,  a_{g_1}\}, \{b_{g_1+1}, ..., b_g\}$.  The union of a component of $a_c - \gamma$ and a component of $a'_1 - \gamma$ cuts off a bubble containing all the arcs $\beta_j$ lying in, say, $W_-$.  A bubble move around a longitude of $\bdd W$ will push $a'_1$ into a new position so that all the $\beta_j$ lie in $W_+$, where the previous argument applies.  The bubble itself may not be standard in the given Powell equivalence class but, invoking the inductive assumption that the Powell Conjecture is true for genus $g - g_1$, the summand contained in the bubble may be made standard without changing the Powell equivalence class.  Then the bubble move itself does not change the Powell equivalence class, completing the proof in this special case.  

{\em General Case}  The general case now proceeds classically, by induction on $|\{a_1, ..., a_{g_1}\} \cap \{a'_1, ..., a'_{g_1}\}|$, the number of arcs in which the two systems of $A$ disks intersect.  Consider an outermost arc of intersection in $a'_1$, say, cutting off an outermost disk $D' \subset a'_1$. .  With no loss assume the arc also lies in $a_1$.  The correct choice of subdisk $D \subset a_1$, when attached to $D'$ along the arc of intersection will give a disk $a_D \subset A$ so that $\{a_D, a_2, ..., a_{g_1}\}$ also satisfies the hypotheses of the theorem.  Since $a_D$ is disjoint from $a_1$ it follows from the special case above that $\{a_D, a_2, ..., a_{g_1}\}$ and $\{a_1, a_2, ..., a_{g_1}\}$ determine the same Powell equivalence class.  Since $|a_D \cap a'_1| < |a_1 \cap a'_1|$ (and no other pairs of disks have an increased number of intersection arcs) the inductive hypothesis implies that  $\{a_D, a_2, ..., a_{g_1}\}$ and $\{a_1, a_2, ..., a_{g_1}\}$ determine the same Powell equivalence class. 
\end{proof}

\section{Towards a proof of the Powell Conjecture}

\subsection{The philosophy} \label{sub:philosophy}  For the purposes of this section, let $T_0$ denote a copy of the standard genus $g \geq 2$ Heegaard surface $T_g$ in $S^3$.  Here is the philosophy behind our  (only partially successful) strategy to prove the Powell Conjecture.  Start with an element of the Goeritz group, represented by a path $\tau_\theta: S^3 \to S^3, 0 \leq \theta \leq 2\pi$ in $\Diff(S^3)$ that starts out as the identity and has $\tau_{2\pi}(T_0) = T_0$.  For brevity denote each surface $\tau_\theta(T_0) \subset S^3$ by $T_\theta$. 

We would like to find such a representative $\tau$ so that for some $0 = \theta_0 < \theta_1 < \theta_2 < .... < \theta_n < 2\pi$ 
\begin{enumerate} 
\item for each $\theta \notin \{\theta_i, 1 \leq i \leq n\}$ we can extract information from the surface $T_\theta \subset S^3$ sufficient to determine a Powell equivalence class of trivializations $h_\theta: (S^3, T_\theta) \to (S^3, T_g)$.
\item ensure that the information is unchanged throughout each interval in $[0, 2\pi] - \{\theta_i, 1 \leq i \leq n\}$ so that $h_\theta \tau_\theta: (S^3, T_0) \to (S^3, T_g)$ is a well-defined Powell equivalence throughout each interval
\item for each $1 \leq i \leq n$ ensure that the information for $\theta_i - \epsilon$ gives the same Powell equivalence class as the information for $\theta_i + \epsilon$ so that $h_\theta \tau_\theta: (S^3, T_0) \to (S^3, T_g)$ is a well-defined Powell equivalence class as we move from one interval to the next.  
\end{enumerate}

It would follow then that  for all $0 \leq \theta \leq 2\pi$, $h_\theta \tau_\theta: (S^3, T_0) \to (S^3, T_g)$ are Powell equivalent.  In particular $h_0 \tau_0$ is Powell equivalent to $h_{2\pi} \tau_{2\pi}$.  But since the Powell equivalence class of $h_\theta$ is determined entirely by the surface $T_\theta$, and $T_{2\pi} = T_0$, it follows that we may take $h_{2\pi} = h_0$ so that $\tau_0$ (the identity) is Powell equivalent to $\tau_{2\pi}$, as required.  

Although this is the philosophy, the outcome of our argument is not so neat.  
Sadly, the information we will be able to extract does not rise to the level (as exhibited, say, in Corollary \ref{cor:weakreduce} above) that is sufficient to determine Powell equivalence class, at least as far as we have been able to determine.   (But it does suffice for the case of genus $3$, see Section \ref{sect:genus3} below.)

\subsection{The complex $2C(T)$}

First we describe the topological information that we will extract.  Thinking of $C(T)$ as the curve complex of $T = T_0$, define a 1-complex $2C(T)$ as follows:  the vertices are disjoint ordered pairs of simple closed curves in $T$ (so each corresponds to an edge in $C(T)$, with an orientation). The 1-simplices in $2C(T)$ are of two types: pairs of pairs $((a, b), (a, b^\pr))$ with the property that the curves $\{a, b, b^\pr\}$ are pairwise disjoint and pairs of pairs $((a, b), (a^\pr, b))$ where the curves $(a, a^\pr, b)$ are pairwise disjoint. (So we can think of each edge as a 2-simplex in $C(T)$ in which two of the edges have been oriented to share a head or a tail.)  We speak of ``shuffling" between $(a, b)$ and $(a, b^\pr)$ and between $(a, b)$ and $(a^\pr, b)$. 

Now let $T_\te = \tau_\te(T_0), \te \in S^1$ be a parameterized Heegaard surface in $S^3$ representing an element of the Goeritz group, as described above.  By a \underln{circle $\gamma$ of weak reductions (cwr)}  representing $T_\te$ we mean an edge path $\gamma_t$, $0 \leq t \leq 2\pi$ in $2C(T)$ so that 
\begin{itemize}
\item each vertex on $\gamma_t$ is a pair $(a, b)$ where $a$ compresses in $A$ and $b$ compresses in $B$, 
\item $\gamma_{2\pi} = \tau_{2\pi}(\gamma_0)$, where $\tau_{2\pi}$ acts via the natural action of $MCG(T)$ on $2C(T)$. 
\end{itemize}

We will show that any such parameterization $\tau_\te$ can be deformed into one that somewhat naturally presents a cwr representing $\tau_\te$.  Specifically, after the deformation, there will be successive values $ \theta_1,  \theta_2,  .... , \theta_n \in (0, 2\pi)$ so that for each $\theta \notin \{\theta_i, 1 \leq i \leq n\}$ there is a pair of weakly reducing disks $(a_\theta, b_\theta)$ associated to the topological surface $T_\theta \subset S^3$ so that
\begin{itemize} 
\item the isotopy class of the pair $(a_\theta, b_\theta)$ is unchanged throughout each interval in $S^1 - \{\theta_i, 1 \leq i \leq n\}$,
\item for each $1 \leq i \leq n$ the pair $(a_{\theta_i + \epsilon}, b_{\theta_i + \epsilon})$ differs from the pair $(a_{\theta_i - \epsilon}, b_{\theta_i - \epsilon})$ by a shuffle 
\end{itemize}

Technically, the cwr is then the pull-back of this sequence of disk pairs under the parameterization $(S^3, T_g) \to (S^3, T_\te)$ that defines $T_\te$.  Before we show how to construct the cwr, we describe how a proof of the following conjecture would then lead to a proof of the Powell Conjecture.

\begin{conj} \label{conj:cwr}  There is a method of associating to any vertex $(a, b)$ in $2C(T)$ for which $a$ bounds a disk in $A$ and $b$ bounds a disk in $B$, a Powell equivalence class of homeomorphisms $h_{(a, b)}(S^3, T) \to (S^3, T_g)$ with these properties:
\begin{itemize}
\item The method is topological.  That is, for any homeomorphism $\sigma: (S^3, T) \to (S^3, T)$ and vertex $(a, b) \in 2C(T)$, $h_{(\sigma(a), \sigma(b))}\sigma = h_{(a, b)}$ and
\item the Powell equivalence class associated to the ends of any edge in $2C(T)$ are the same.  
\end{itemize} 
\end{conj}

Combining Conjecture \ref{conj:cwr} with the construction that precedes it, we get a sequence of homeomorphisms $h_{(a_0, b_0)}, h_{(a_1, b_1)}, ..., h_{(a_n, b_n)}: T \to T_g$ which all have the same Powell equivalence class and for which $(a_n, b_n) = (\tau_{2\pi}(a_0),\tau_{2\pi}(b_0))$.  This implies that $h_{(a_0, b_0)}\tau_{2\pi}$ is Powell equivalent to $h_{(\tau_{2\pi}(a_0),\tau_{2\pi}(b_0))}\tau_{2\pi} = h_{(a_0, b_0)}$ so $\tau_{2\pi}$ is Powell equivalent to the identity, as required.

\subsection{The Rieck background and its refinement}  \label{sub:rieck}

Recall the classical sweep-out technology applicable to any Heegaard splitting of a closed $3$-manifold $M = A \cup_T B$ (see \cite{RS}):  Pick a spine in each handlebody $A$, $B$, that is a $1$-complex $X$ in $A$ (say) whose complement is homeomorphic to $T \times (0, 1)$, with $T \times \{1\}$ corresponding to $\bdd A$.  This gives rise to  a mapping cylinder structure on $A$, $A \cong \de A \times [0,1] /$\small$(a \times 0) \equiv f(a)$\normalsize, some $f: \de A \ra X$.  In its classical application, the mapping cylinder structures on $A$ and $B$ can be combined to parameterize the entire complement of the spines in $M$ as $T \times (-1, 1)$ describing how most of $M$ is swept-out by copies of $T$.

In the case that $M = S^3$ there is another natural sweep-out (actually the genus $0$ version of the sweep-out just described).  Viewing $S^3$ as the standard $3$-sphere in $4$-space, a height function in $R^4$ describes a sweep-out of $S^3$ from south-pole to north-pole.  That is, we can view $S^3$ also as $S^2 \times [-1, 1]$ with $S^2 \times \{0\}$ crushed to the south pole and $S^2 \times \{1\}$ crushed to the north pole.  

In  \cite{R} (based on arguments in \cite{RS}) Rieck proves Waldhausen's theorem by comparing these two ``sweep-outs" of $S^3$ by surfaces, one parameter $s \in [-1, 1]$ for the sweep out by spheres $S$ and one parameter $t \in [-1, 1]$ for the sweep-out by a genus $g \geq 2$ Heegaard surface $T$.  A ``graphic" $\Ggg$ of this $(t, s)$-square is analyzed to find a weak reduction of $T$, that is a pair of compressing disks, one in $A$ and the other in $B$, whose boundaries are disjoint in $T$.  Then \cite{CG} implies that the Heegaard splitting is reducible, and that finishes the proof by induction.  

The graphic consists of open regions $R_i$ where $S_s$ and $T_t$ intersect transversely, edges or ``walls" where the two have a tangency, and cusp points where two types of tangencies cancel. As argued in \cite{RS} only domain walls corresponding to saddle tangencies need to be tracked. Cusps and tangencies of index 2 or 0 can be erased as they amount only to births/deaths of \underln{inessential} simple closed curves of intersection in $S_s \cap T_t$. The most interesting event which occurs are transverse crossings of saddle walls; at this point two independent saddle tangencies occur.  It will be useful to very briefly review Rieck's analysis of this graphic:

Each region $R_i \subset I \times I \backslash \Ggg$ is labeled by $I$ or $E$: $I$ if every component of $S_s \cap T_t$, $(t, s) \in R_i$ is inessential in $T$, $E$ otherwise. Call an $I$ region \underln{green}, labeled $I_g$, if $T_t \cap (\bigcup_{s^\pr \geq s} S_{s^\pr}) \subset T_t \backslash$disks, that is most of $T_t$ lies north of $S_s$. Similarly we call an $I$ region \underln{yellow}, labeled $I_y$, if $T_t \cap \bigcup_{s^\pr < s} S_{s^\pr}) \subset T_t \backslash$ disks, that is most of $T_t$ lies south of $S_s$. \cite[Lemma/Definition 3.4]{R} asserts that every $I$-region is labeled exactly once as $I_g$ or $I_y$. Furthermore, \cite[Proposition 3.3]{R} states that no two $I$ regions bearing different color labels can touch, even at a corner.  (This is where requiring genus $g \geq 2$ comes in; see for example Lemma \ref{lemma:planar} in the Appendix.)

It is clear that for $s$ sufficiently near -1, $(t, s)$ must lie in an $I_g$ region (since most of $T_t$ lies north of $S_s$) and that for $s$ sufficiently near +1, $(t, s)$ must lie in an $I_y$ region.  It follows that between and $I_g$ and $I_y$ region there is an entire strip of $I \times I$, running from the side $t = -1$ to the side $t = 1$, in which each region is labelled $E$.  Moreover, when $t$ is near $-1$, $T_t$ is close to the spine of $A$, so the essential curves bound disks in $A$ and when $t$  is near $+1$, $T_t$ is close to the spine of $B$, so the essential curves bound disks in $B$.  It follows that there must be some region (or two adjacent regions) in such a strip in which both sorts of disks occur, and this provides a weak reduction.  

Rieck's result can be refined.  In the Appendix \ref{appendix} we show that the lowest of the strips that appear in Rieck's argument is actually monotonic in $t$.  By this we mean that there is a function $r: [-1, 1] \to [-1, 1]$ whose graph lies in the $1$-skeleton of the reduced graphic and, for small $\eee$, $\{(t, r(t) + \eee)\} | t \in [-1,1]\}$ lies entirely inside the lowest strip.  To put it another way, the graph of $r + \eee$ in $I \times I$ is transverse to the reduced graphic and, for any $(t, s)$ lying in a region that the graph intersects, among the circles in $S_s \cap T_t$, there is at least one circle that is essential in $T_t$.  This fact has no real importance in Rieck's argument, but its analogue in our context will be quite useful (though not mathematically essential).  

\subsection{Adding the parameter $\te$} \label{sub:theta}

The fundamental idea now will be to add a third parameter to Rieck's proof, namely the parameter $\te$.  We first need to show that the choice of spines in $A$ and $B$ above is unimportant, even in the general context of  a closed manifold $M = A \cup B$.  One can think of a spine for $A$ (or $B$) as a 1-complex $X \subset \text{int } A$ together with a dual cell structure on $A$: a proper disk associated with (and normal to) each 1-simplex of $X$ and a 3-ball containing each vertex of $X$. But such a presentation of the spine involves \underln{no} choice in the sense that its space of parameters is contractible.  The argument for this has two parts:  It follows from \cite{Mc} that the choice of disks defining the spine is unimportant since the disk space of $A$ is contractible.  (Here the disk space is the simplicial complex whose $n$-cells are $n+1$ pairwise disjoint properly embedded disks, which together divide $A$ into balls.) Once these disks are chosen, Hatcher has shown that the exact placement of the disks in $A$, and then the ensuing parameterization of the complementary $3$-balls, also involves no choice. Putting together these results it follows that once a Heegaard surface $T \subset M$ is determined, $M$ is foliated by a family of surfaces $M = \bigcup_{t \in [-1, 1]} T_t$ which degenerate at $t \in \{-1, +1\}$.  Here $T_0$ is of course the original Heegaard surface $T = \de A = \de B$.  Moreover, this foliation is canonical, up to choices coming from a weakly contractible parameter space. The disk complex is contractible and the corresponding function spaces whose simplicial realization is the disk complex is therefore weakly contractible.

Now add the circular parameter $\te \in [0, 2\pi] /$\small$0 \equiv 2\pi$ \normalsize and consider the family of Heegaard surfaces $T_\te \subset S^3$. Choose spines for $A_0$ and $B_0$, propagate them along $\te$ by isotopy until at $2\pi - \epsilon$ we see that the initial and final spines are not matching up. However using just the $\pi_1$-aspect of ''canonical" above we may isotope the spines to match at $2\pi = 0$. This yields a $\te$-parameter family of singular foliations $\ppap$, $(t, \te) \in [-1, 1] \times [0, 2\pi] /$\small$0 \equiv 2\pi \coloneqq I \times S^1$\normalsize.

\begin{thm}  \label{thm:cwr}  
  Let $T_\te$ be a loop of genus $g$ Heegaard surfaces, $g \geq 2$ as described above.  There is a 
  circle $\gamma$ of weak reductions (cwr) representing $T_\te$.
\end{thm}

\begin{proof}  We study the 3-parameter family $\ppap \cap S_s$, $(t, s, \te) \in [-1, 1] \times [-1, 1] \times [0, 2\pi] /$\small$0 \equiv 2\pi \coloneqq [-1, 1] \times [-1, 1] \times S^1$\normalsize. These intersections may be regarded as the level set at height = $s$ of a 2-parameter family of functions on $T$. All told this means we are looking for a 3-parameter family of germs of smooth functions $(R^2, 0) \ra (R, 0)$. According to Thom's Jet-transversality Theorem any such family can be perturbed a generic one with local simulation types as listed below. In these dimensions the generic local germs \footnote{All other germ types such as hyperbolic umbilic, $f(x, y) = x^3 + y^3$, have codimensions $\geq 4$ and need not be considered.} \cite{HW} are represented by:

\begin{table}[H]
  \centering
  \setlength{\tabcolsep}{12pt}
  \begin{tabular}{lllll}
    1.                & $f(x, y) = x$         &                                     & codim = 0 \\
    2a.               & $f(x, y) = x^2 + y^2$ & source / sink                       & codim = 1 \\
    \hspace{.25em} b. & $f(x, y) = x^2 - y^2$ & saddle                              & codim = 1 \\
    3a.               & $f(x, y) = x^3 + y^2$ & birth/death of $(1, 2)$-handle pair & codim = 2 \\
    \hspace{.25em} b. & $f(x, y) = x^3 - y^2$ & birth/death of $(1, 2)$-handle pair & codim = 2 \\
    4a.               & $f(x, y) = x^4 + y^2$ & dovetail singularity                & codim = 3 \\
    \hspace{.25em} b. & $f(x, y) = x^4 - y^2$ & dovetail singularity                & codim = 3 \\
    \rule{0pt}{2.5ex} & & ( \begin{tikzpicture}
      \draw (-0.5,0.5) to [out=0] (0.1,0.2);
      \draw (-0.3,0.2) to [in=180] (0.3,0.5);
      \draw (-0.3,0.2) to (0.1,0.2);
      \draw [->] (0.4,0.3) -- (0.7,0.3);
      \draw (0.8,0.5) to [out=0] (1.2,0.2);
      \draw (1.2,0.2) to [in=180] (1.6,0.5);
      \draw [->] (1.7,0.3) -- (2,0.3);
      \draw (2.1,0.5) -- (2.3,0.4);
      \draw (2.7,0.4) -- (2.9,0.5);
      \draw (2.3,0.4) to [out=-60,in=180] (2.5,0.2) to [out=0,in=240] (2.7,0.4);
    \end{tikzpicture} ) & \\
  \end{tabular}
\end{table}

These are the germs. We also may invoke genericity to ensure that where a single function on $T$ restricts to singular germs on several different points of $T$, the local unfoldings are transverse. These is no difficulty making this precise as the stratification of the unfoldings obey the two Whitney conditions \cite{W} are are in fact piecewise smooth submanifolds of the parameter space.

So what does the singular set of a \underln{graphic} in $[-1, 1] \times [-1, 1] \times S^1$ look like? It consists of 2D-manifold sheets of two types - saddle and sink - meeting along birth/death 1D strata, which are allowed to have dovetail singularities at isolated points. The 1D and 2D strata may pass transversely through each other at points and the 2D strata may have transverse 1-manifolds of intersection and isolated standard triple points.

The picture we must study is much simpler. Just as in \cite{RS}, the 2D sheets labeled by source/sink tangencies may be discarded forming the \underln{reduced graphic}, as crossing these walls only changes the intersection $T \cap S$ by inessential simple closed curves. The remaining saddle-labeled walls (it does not matter for connectivities, but for convenience take their closures) divide $[-1, 1] \times [-1, 1] \times S^1$ into complementary open regions $\{R_i\}$, each of which has a constant topological pattern $T \cap S$, up to birth and death of inessential sccs. Thus the 3D reduced graphic $G_3$ is a straightforward generalization of the 2D one: 2-manifold sheets (perhaps with borders) crossing in double curves and triple points.

\subsection{The graphic on an annular surface} \label{sub:annular}

After having defined the 3D reduced graphic, we now replace it with a 2D graphic by restricting it to the sub-surface $\Sigma \subset [-1, 1] \times [-1, 1] \times S^1$, where $\Sigma$ is the graph of the function described in Proposition \ref{prop:graphsum}.  Its structure as a graph provides $\Sigma$ with a natural diffeomorphism to an annulus $[-1, 1] \times S^1$, parameterized by $(t, \te)$, and that is how we will view it. This structure as a graph allows us, when thinking of the intersection $\ppap \cap S_s$ that corresponds to a point in $\Sigma$, to take the sphere as fixed, so henceforth we drop the subscript $s$.   (When there is little risk of confusion, we will also drop the subscript on $T$.)  Each region of the graphic in $\Sigma$ corresponds to a transverse intersection of $\ppap$ with $S$ in which some of the curves of intersection are essential in $\ppap$ and bound disks in $A_{t, \te}$ or $B_{t, \te}$, perhaps both.  For regions near $t = -1$, all such disks lie in $A_{t, \te}$ (since $T$ is near the spine of $A_{t, \te}$, which we may take transverse to the height function $s$) and we label these regions $A$.  Similarly, near $t = +1$, all such disks lie in $B_{t, \te}$ and we label such regions $B$.  

Let us establish the following labeling and ``artistic" convention for the general region $R$ of our annular graphic.  If $\ppap \cap S$ contains a scc $a$ essential in $\ppap$ and compressing in the handlebody $A_{t, \te}$ label the region containing $(t, \te)$ by $A$. (Do this regardless of whether $\ppap \cap S$ contains an essential scc compressing in $B_{t, \te}$). If $\ppap \cap S$ does not contain any essential scc compressing in $A_{t, \te}$, but does contain a scc compressing in $B$, label that Region $B$.

\underln{Artistic convention}: If labels $A$ and $B$ alternate around a saddle wall double point render the boundary between $A$-regions and $B$-regions to be a 1-manifold favoring $A$. We label this 1-manifold $M$, it separates $A$ regions from $B$ regions. We recall both this labeling and rendering convention with the phrase: ``favor $A$."
 
 Since near one boundary component of  the annulus $\Sigma$ the regions are all labelled $A$ and at the other they are all labelled $B$, there is at least one component of $M$ that is essential in $\Sigma$.  Such a component will be homotopic in $\Sigma$, fixing a point, to a core curve $\{pt\} \times S^1 \subset [-1, 1] \times S^1$.  Here then is the plan for the proof of Theorem \ref{thm:cwr}.  We will show
\begin{itemize}
\item Each component of $M$ in $\Sigma$ naturally describes an edge-path in $2C(T)$ through vertices $(a, b)$ in which the first curve bounds an essential disk in $A$ and the second an essential disk in $B$
\item For any component of $M$ that is essential in $\Sigma$ the associated edge-path is a cwr representing $\tau$.
\end{itemize}

\begin{figure}[ht]
  \centering
  \begin{tikzpicture}
    \draw (0, 0) -- (4, 4);
    \draw (0, 4) -- (4, 0);
    \node at (2, 0.75) {\large A};
    \node at (0.5, 2) {\large B};
    \node at (3.5, 2) {\large B};
    \node at (2, 3.25) {\large A};

    \node at (6, 2) {\Huge $\longrightarrow$};
    \node at (5.93, 2.35) {\small render};

    \draw (8, 0) -- (9.4, 1.4);
    \draw (8, 4) -- (9.4, 2.6);
    \draw (9.4, 1.4) to [out = 45, in = 315] (9.4, 2.6);
    \draw (12, 0) -- (10.6, 1.4);
    \draw (12, 4) -- (10.6, 2.6);
    \draw (10.6, 1.4) to [out = 135, in = 225] (10.6, 2.6);
    \node at (10, 0.75) {\large A};
    \node at (8.5, 2) {\large B};
    \node at (11.5, 2) {\large B};
    \node at (10, 3.25) {\large A};
  \end{tikzpicture}
  \caption{}
\end{figure}

The crux is to understand the intersection locus where two saddle walls cross, what the four possible resolutions of the two saddles look like, and which curves among these resolution could be labeled $a$ or $b$. (Again a curve is labeled $a$ (resp $b$) if it is essential in $T$ and compresses in $A$ (resp $B$).)

The singular pattern $\xi$ associated with a saddle wall crossing is present simultaneously in the sphere $S$ and the surface $T$. Thus the pattern $\xi$ as a 4-valent graph is the same in $S$ and $T$. Furthermore the tangential information is also preserved so at a saddle tangency we know which legs are opposite (this is not quite the information of a cyclic ordering; it might be called a "dihedral ordering" as the information is a coset of $S(4) / D(4)$). But interestingly, it is \underln{not} true that the pairs $(\mathcal{N}_S(\xi), \xi) \subset S$ and $(\mathcal{N}_T(\xi), \xi) \subset T$ are necessarily homeomorphic.  (Here $\mathcal{N}_S$ and $\mathcal{N}_T$ denote neighborhood in $S$ and $T$ respectively.) The important example is Figure \ref{fig:toralsaddle}; it is a torus of revolution intersected with a sloping plane to produce two circles of intersection. In the plane they have one positive and one negative intersection, whereas in the torus they have two positive intersections.

\begin{figure}[H]
  \centering
  \begin{tikzpicture}
    \draw  (0,0) ellipse (3.5 and 1.5); 
    \draw (-0.8,0.1) arc (180:360:0.8 and 0.2); 
    \draw [dashed] (-0.73,0.02) arc (180:0:0.73 and 0.15); 

    \draw (0.06,-0.16) to [out=-45,in=225] (1.8,0);
    \draw (0.061, -0.161) to [out=135,in=0] (0, -0.1);
    \draw (1.8,0) to [out=45,in=10] (-1.8,1.29);
    \draw [dashed] (-0.075,-0.175) to [out=225,in=-45] (-1.8,0);
    \draw [dashed] (-0.061, -0.161) to [out=45,in=180] (0, -0.1);
    \draw [dashed] (-1.8,0) to [out=135,in=195] (-1.8,1.29);

    \draw (0.06,0.22) to [out=45,in=135] (1.8,0);
    \draw (0.06,0.22) to [out=225,in=0] (0, 0.165);
    \draw (1.8,0) to [out=-45,in=-10] (-1.8,-1.29);
    \draw [dashed] (-0.09,0.25) to [out=135,in=45] (-1.8,0);
    \draw [dashed] (-0.06,0.22) to [out=-45,in=180] (0,0.165);
    \draw [dashed] (-1.8,0) to [out=225,in=165] (-1.8,-1.29);

    \node at (2.7,0) {\footnotesize $\xi\subset$ Torus};
    \node at (0,-2) {(plane not rendered)};
  \end{tikzpicture}
  \caption{} \label{fig:toralsaddle}
\end{figure}

For our analysis we need to list all possible \underln{pairs} $(\mathcal{N}(\xi), \xi)$ up to homeomorphism. Let us start by enumerating the possibilities on $S$.  There are 4:

\begin{figure}[ht]
  \centering
  \begin{tikzpicture}
    \draw (-3,0) -- (-3,-1.5) -- (-0.5,-1.5) -- (-0.5,0) -- cycle;
    \draw (-4.5,0) -- (-4.5,-1.5) -- (-6,-1.5) -- (-6,0) -- cycle;
    \draw (1,0) -- (3.5,0) -- (3.5,-1.5) -- (1,-1.5) -- cycle;
    \draw (5, 0) -- (7.5,0) -- (7.5,-1.5) -- (5,-1.5) -- cycle;
    \draw  (5.8,-0.75) ellipse (0.6 and 0.4);
    \draw  (6.7,-0.75) ellipse (0.6 and 0.4);
    \draw (1.6,-1.2) -- (2.9,-1.2);
    \draw (1.6,-1.2) to [out=180,in=180] (1.6,-0.6);
    \draw (2.9,-1.2) to [out=0,in=0] (2.9,-0.6);
    \draw (2.9,-0.6) -- (2.59,-0.6);
    \draw (2.6,-0.6) to [out=180,in=90] (2.4,-0.75);
    \draw (2.4,-0.74) to [out=270,in=180] (2.505,-0.85);
    \draw (2.5,-0.85) to [out=0,in=270] (2.6,-0.74);
    \draw (2.6,-0.75) to [out=90,in=0] (2.4,-0.6);
    \draw (1.9,-0.6) -- (1.6,-0.6);
    \draw (1.9,-0.6) to [out=0,in=270] (2.1,-0.45);
    \draw (2.1,-0.45) to [out=90,in=0] (2,-0.35);
    \draw (2,-0.35) to [out=180,in=90] (1.9,-0.45);
    \draw (1.9,-0.45) to [out=270,in=180] (2.1,-0.6);
    \draw (2.1,-0.6) -- (2.4,-0.6);
    \draw (-2.4,-1.2) -- (-1.1,-1.2);
    \draw (-2.4,-1.2) to [out=180,in=180] (-2.4,-0.6);
    \draw (-1.1,-1.2) to [out=0,in=0] (-1.1,-0.6);
    \draw (-1.1,-0.6) -- (-1.4,-0.6);
    \draw (-1.6,-0.6) -- (-1.9,-0.6);
    \draw (-2.1,-0.6) to [out=0,in=270] (-1.9,-0.45);
    \draw (-1.9,-0.45) to [out=90,in=0] (-2,-0.35);
    \draw (-2,-0.35) to [out=180,in=90] (-2.1,-0.45);
    \draw (-2.1,-0.45) to [out=270,in=180] (-1.9,-0.6);
    \draw (-1.6,-0.6) to [out=0,in=270] (-1.4,-0.45);
    \draw (-1.4,-0.45) to [out=90,in=0] (-1.5,-0.35);
    \draw (-1.5,-0.35) to [out=180,in=90] (-1.6,-0.45);
    \draw (-1.6,-0.45) to [out=270,in=180] (-1.4,-0.6);
    \draw (-2.4,-0.6) -- (-2.1,-0.6);
    \draw (-5.5,-0.25) to [out=180,in=90] (-5.75,-0.45);
    \draw (-5.5,-0.65) to [out=180,in=270] (-5.75,-0.45);
    \draw (-5.5,-0.25) to [out=0,in=115] (-5.25,-0.45);
    \draw (-5.5,-0.65) to [out=0,in=245] (-5.25,-0.45);
    \draw (-5.25,-0.45) to [out=65,in=180] (-5,-0.25);
    \draw (-5.25,-0.45) to [out=-65,in=180] (-5,-0.65);
    \draw (-5,-0.65) to [out=0,in=270] (-4.75,-0.45);
    \draw (-5,-0.25) to [out=0,in=90] (-4.75,-0.45);

    \draw (-5.5,-0.85) to [out=180,in=90] (-5.75,-1.05);
    \draw (-5.5,-1.25) to [out=180,in=270] (-5.75,-1.05);
    \draw (-5.5,-0.85) to [out=0,in=115] (-5.25,-1.05);
    \draw (-5.5,-1.25) to [out=0,in=245] (-5.25,-1.05);
    \draw (-5.25,-1.05) to [out=65,in=180] (-5,-0.85);
    \draw (-5.25,-1.05) to [out=-65,in=180] (-5,-1.25);
    \draw (-5,-1.25) to [out=0,in=270] (-4.75,-1.05);
    \draw (-5,-0.85) to [out=0,in=90] (-4.75,-1.05);

    \node at (-5.2,-1.9) {1};
    \node at (-1.7,-1.9) {2};
    \node at (2.3,-1.9) {3};
    \node at (6.3,-1.9) {4};
  \end{tikzpicture}
  \caption{}
\end{figure}

These four possible $\xi$s can yield neighborhood pairs on $T$ as shown in Figure \ref{fig:xiT}.  (The small circles in the figures represent boundary components of the neighborhood, typically essential circles in $T$.)

\begin{figure}[H]
  \centering
  \begin{tikzpicture}
    \draw (-6,1.5) -- (-6,0) -- (-3.5,0) -- (-3.5,1.5) -- cycle;
    \draw (-4.5,4) -- (-4.5,2.5) -- (-6,2.5) -- (-6,4) -- cycle;
    \draw (-2,1.5) -- (0.5,1.5) -- (0.5,0) -- (-2,0) -- cycle;
    \draw (-6,-3.75) -- (-3.5,-3.75) -- (-3.5,-5.25) -- (-6,-5.25) -- cycle;
    \draw  (-5.2,-4.5) ellipse (0.6 and 0.4);
    \draw  (-4.3,-4.5) ellipse (0.6 and 0.4);
    \draw (-1.4,0.3) -- (-0.1,0.3);
    \draw (-1.4,0.3) to [out=180,in=180] (-1.4,0.9);
    \draw (-0.1,0.3) to [out=0,in=0] (-0.1,0.9);
    \draw (-0.1,0.9) -- (-0.41,0.9);
    \draw (-0.4,0.9) to [out=180,in=90] (-0.6,0.75);
    \draw (-0.6,0.76) to [out=270,in=180] (-0.495,0.65);
    \draw (-0.5,0.65) to [out=0,in=270] (-0.4,0.76);
    \draw (-0.4,0.75) to [out=90,in=0] (-0.6,0.9);
    \draw (-1.1,0.9) -- (-1.4,0.9);
    \draw (-1.1,0.9) to [out=0,in=270] (-0.9,1.05);
    \draw (-0.9,1.05) to [out=90,in=0] (-1,1.15);
    \draw (-1,1.15) to [out=180,in=90] (-1.1,1.05);
    \draw (-1.1,1.05) to [out=270,in=180] (-0.9,0.9);
    \draw (-0.9,0.9) -- (-0.6,0.9);
    \draw (-6,-1) -- (-3.5,-1) -- (-3.5,-2.5) -- (-6,-2.5) -- cycle;
    \draw (-5.4,-2.2) -- (-4.1,-2.2);
    \draw (-5.4,-2.2) to [out=180,in=180] (-5.4,-1.6);
    \draw (-4.1,-2.2) to [out=0,in=0] (-4.1,-1.6);
    \draw (-4.1,-1.6) -- (-4.41,-1.6);
    \draw (-4.4,-1.6) to [out=180,in=90] (-4.6,-1.75);
    \draw (-4.6,-1.74) to [out=270,in=180] (-4.495,-1.85);
    \draw (-4.5,-1.85) to [out=0,in=270] (-4.4,-1.74);
    \draw (-4.4,-1.75) to [out=90,in=0] (-4.6,-1.6);
    \draw (-5.1,-1.6) -- (-5.4,-1.6);
    \draw (-5.1,-1.6) to [out=0,in=270] (-4.9,-1.45);
    \draw (-4.9,-1.45) to [out=90,in=0] (-5,-1.35);
    \draw (-5,-1.35) to [out=180,in=90] (-5.1,-1.45);
    \draw (-5.1,-1.45) to [out=270,in=180] (-4.9,-1.6);
    \draw (-4.9,-1.6) -- (-4.6,-1.6);
    \draw (-5.4,0.3) -- (-4.1,0.3);
    \draw (-5.4,0.3) to [out=180,in=180] (-5.4,0.9);
    \draw (-4.1,0.3) to [out=0,in=0] (-4.1,0.9);
    \draw (-4.1,0.9) -- (-4.4,0.9);
    \draw (-4.6,0.9) -- (-4.9,0.9);
    \draw (-5.1,0.9) to [out=0,in=270] (-4.9,1.05);
    \draw (-4.9,1.05) to [out=90,in=0] (-5,1.15);
    \draw (-5,1.15) to [out=180,in=90] (-5.1,1.05);
    \draw (-5.1,1.05) to [out=270,in=180] (-4.9,0.9);
    \draw (-4.6,0.9) to [out=0,in=270] (-4.4,1.05);
    \draw (-4.4,1.05) to [out=90,in=0] (-4.5,1.15);
    \draw (-4.5,1.15) to [out=180,in=90] (-4.6,1.05);
    \draw (-4.6,1.05) to [out=270,in=180] (-4.4,0.9);
    \draw (-5.4,0.9) -- (-5.1,0.9);
    \draw (-2,-1) -- (-2,-2.5) -- (0.5,-2.5) -- (0.5,-1) -- cycle;
    \draw (-1.4,-2.2) -- (-0.1,-2.2);
    \draw (-1.4,-2.2) to [out=180,in=180] (-1.4,-1.6);
    \draw (-0.1,-2.2) to [out=0,in=0] (-0.1,-1.6);
    \draw (-0.1,-1.6) -- (-0.4,-1.6);
    \draw (-0.6,-1.6) -- (-0.9,-1.6);
    \draw (-1.1,-1.6) to [out=0,in=270] (-0.9,-1.45);
    \draw (-0.9,-1.45) to [out=90,in=0] (-1,-1.35);
    \draw (-1,-1.35) to [out=180,in=90] (-1.1,-1.45);
    \draw (-1.1,-1.45) to [out=270,in=180] (-0.9,-1.6);
    \draw (-0.6,-1.6) to [out=0,in=270] (-0.4,-1.45);
    \draw (-0.4,-1.45) to [out=90,in=0] (-0.5,-1.35);
    \draw (-0.5,-1.35) to [out=180,in=90] (-0.6,-1.45);
    \draw (-0.6,-1.45) to [out=270,in=180] (-0.4,-1.6);
    \draw (-1.4,-1.6) -- (-1.1,-1.6);
    \draw (-5.5,3.75) to [out=180,in=90] (-5.75,3.55);
    \draw (-5.5,3.35) to [out=180,in=270] (-5.75,3.55);
    \draw (-5.5,3.75) to [out=0,in=115] (-5.25,3.55);
    \draw (-5.5,3.35) to [out=0,in=245] (-5.25,3.55);
    \draw (-5.25,3.55) to [out=65,in=180] (-5,3.75);
    \draw (-5.25,3.55) to [out=-65,in=180] (-5,3.35);
    \draw (-5,3.35) to [out=0,in=270] (-4.75,3.55);
    \draw (-5,3.75) to [out=0,in=90] (-4.75,3.55);

    \draw (-5.5,3.15) to [out=180,in=90] (-5.75,2.95);
    \draw (-5.5,2.75) to [out=180,in=270] (-5.75,2.95);
    \draw (-5.5,3.15) to [out=0,in=115] (-5.25,2.95);
    \draw (-5.5,2.75) to [out=0,in=245] (-5.25,2.95);
    \draw (-5.25,2.95) to [out=65,in=180] (-5,3.15);
    \draw (-5.25,2.95) to [out=-65,in=180] (-5,2.75);
    \draw (-5,2.75) to [out=0,in=270] (-4.75,2.95);
    \draw (-5,3.15) to [out=0,in=90] (-4.75,2.95);

    \draw (-1.6,-4) arc (138:42:1.75 and 0.75); 
    \draw (-1.6,-5) arc (222:318:1.75 and 0.75); 
    \draw (-0.7,-4.45) arc (180:360:0.4 and 0.1); 
    \draw [dashed] (-0.665,-4.49) arc (180:0:0.365 and 0.075); 

    \draw (-0.27,-4.58) to [out=-45,in=225] (0.6,-4.5);
    \draw (-0.2695,-4.5805) to [out=135,in=0] (-0.3,-4.55);
    \draw (0.6,-4.5) to [out=45,in=10] (-1.2,-3.855);
    \draw [dashed] (-0.3375,-4.5875) to [out=225,in=-45] (-1.2,-4.5);
    \draw [dashed] (-0.3305,-4.5805) to [out=45,in=180] (-0.3,-4.55);
    \draw [dashed] (-1.2,-4.5) to [out=135,in=195] (-1.2,-3.855);

    \draw (-0.27,-4.39) to [out=45,in=135] (0.6,-4.5);
    \draw (-0.27,-4.39) to [out=225,in=0] (-0.3,-4.4175);
    \draw (0.6,-4.5) to [out=-45,in=-10] (-1.2,-5.145);
    \draw [dashed] (-0.345,-4.375) to [out=135,in=45] (-1.2,-4.5);
    \draw [dashed] (-0.33,-4.39) to [out=-45,in=180] (-0.3,-4.4175);
    \draw [dashed] (-1.2,-4.5) to [out=225,in=165] (-1.2,-5.145);

    \draw  (-1.6,-4.5) ellipse (0.1 and 0.5); 
    \draw  (1,-4.5) ellipse (0.1 and 0.5); 
    \draw  (-2,-3.5) rectangle (1.5,-5.5);

    \draw  (-5.5,3.54) ellipse (0.07 and 0.05);
    \draw  (-4.98,3.54) ellipse (0.07 and 0.05);
    \draw  (-5.5,2.94) ellipse (0.07 and 0.05);
    \draw  (-4.98,2.94) ellipse (0.07 and 0.05);

    \draw  (-5,1.04) ellipse (0.03 and 0.04);
    \draw  (-4.5,1.04) ellipse (0.03 and 0.04);
    \draw  (-4.75,0.6) ellipse (0.4 and 0.1);

    \draw  (-1,-1.47) ellipse (0.03 and 0.04);
    \draw  (-0.5,-1.47) ellipse (0.03 and 0.04);
    \draw  (-0.75,-1.9) ellipse (0.4 and 0.1);

    \draw  (-5,-1.46) ellipse (0.03 and 0.04);
    \draw  (-4.5,-1.74) ellipse (0.03 and 0.04);
    \draw  (-4.75,-2) ellipse (0.4 and 0.1);

    \draw  (-1,1.04) ellipse (0.03 and 0.04);
    \draw  (-0.5,0.76) ellipse (0.03 and 0.04);
    \draw  (-0.75,0.5) ellipse (0.4 and 0.1);

    \draw  (-5.3,-4.5) ellipse (0.25 and 0.15);
    \draw  (-4.2,-4.5) ellipse (0.25 and 0.15);
    \draw  (-4.75,-4.5) ellipse (0.05 and 0.07);

    \draw [thick, ->] (-8.4,3.2) -- (-6.8,3.2);
    \draw [thick, ->] (-8.4,0.7) -- (-6.8,0.7);
    \draw [thick, ->] (-8.4,-1.8) -- (-6.8,-1.8);
    \draw [thick, ->] (-8.4,-4.5) -- (-6.8,-4.5);
    \node at (-9.2,3.2) {1};
    \node at (-9.1,0.7) {2};
    \node at (-9.2,-1.8) {3};
    \node at (-9.2,-4.5) {4};
    \node at (-2.7,0.7) {or};
    \node at (-2.7,-1.8) {or};
    \node at (-2.7,-4.5) {or};
    \node at (-2.6,3.2) {(unique)};
    \node at (-0.3,-6) {Case 5};
  \end{tikzpicture}
  \caption{}  \label{fig:xiT}
\end{figure}

So, overall, on $T$ we have 5 cases to consider: 1, 2, 3, 4, and 5 above.

Figure \ref{fig:resolved} displays the four local resolutions within $\mathcal{N}(\xi)$ of $\xi$ in each of the 5 cases.

\begin{figure}[ht]
  \centering
  \begin{tikzpicture}
    \draw (0,1)--(0,5);
    \draw (-2,3) -- (2,3);
    \draw  (-1.5,4.25) circle (0.1);
    \draw  (-0.75,4.25) circle (0.1);
    \draw  (-1.125,3.75) ellipse (0.5 and 0.125);
    \draw  (1.6,4.4) circle (0.1);
    \draw [rounded corners] (0.75,4.5) -- (0.75,3.5) -- (1.75,3.5);
    \draw [rounded corners] (1,4.5) -- (1,3.7) -- (1.75,3.7);
    \draw (0.75,4.5) to [out=90,in=90] (1, 4.5);
    \draw (1.75,3.7) to [out=0, in=0] (1.75, 3.5);
    \draw  (-1.4,1.7) circle (0.1);
    \draw [rounded corners] (-0.6,1.6) -- (-0.6,2.6) -- (-1.6,2.6);
    \draw [rounded corners] (-1.6,2.4) -- (-0.8,2.4) -- (-0.8,1.6);
    \draw (-1.6,2.6) to [out=180,in=180] (-1.6,2.4);
    \draw (-0.6,1.6) to [out=270,in=270] (-0.8,1.6);
    \draw [rounded corners] (0.4,2.6)  -- (0.4,1.6) -- (1.8,1.6) -- (1.8,2.6) -- (1.5,2.6) -- (1.5,1.9) -- (0.7,1.9) -- (0.7,2.6) -- cycle;
    \draw [<->, thick] (-0.3,2.7) -- (0.3,3.3);

    \draw (3.5,3)--(7.5,3);
    \draw (5.5,1)--(5.5,5);
    \draw [rounded corners] (3.7,2.1) -- (3.7,1.6) -- (5.1,1.6) -- (5.1,2.3) -- (4.8,2.3) -- (4.8,1.8) -- (4.56,1.8) -- (4.56,2.1) -- cycle;
    \draw  (4,2.4) circle (0.1);
    \draw [rounded corners] (4.9,3.5) -- (3.7,3.5) -- (3.7,4.1) -- (4.9,4.1);
    \draw (4.9,3.5) to [out=0,in=270] (5.15,3.8);
    \draw (5.15,3.8) to [out=90,in=0] (4.9,4.1);
    \draw  (4.9,3.9) circle (0.1);
    \draw  (3.9,4.4) circle (0.1);
    \draw [rounded corners] (5.9,4.5)  -- (5.9,3.5) -- (7.3,3.5) -- (7.3,3.9) -- (6.3,3.9) -- (6.3,4.5) -- cycle;
    \draw  (7.1,3.7) circle (0.07);
    \draw [rounded corners] (5.9,2.1) -- (5.9,1.6) -- (7.3,1.6) -- (7.3,2.1) -- (7,2.1) -- (7,1.85) -- (6.65,1.85) -- (6.65,2.1) -- (6.38,2.1) -- (6.38,2.4) -- (6.15,2.4) -- (6.15,2.1) -- cycle;
    \draw [<->, thick] (5.1,3.4) -- (5.8,2.7);

    \draw (-3.5,3)--(-7.5,3);
    \draw (-5.5,1) -- (-5.5,5);
    \draw (-0.75,-1.5)--(-4.75,-1.5);
    \draw (-2.75,-3.5) -- (-2.75,0.5);
    \draw (-4.5,-0.25) to [out=180,in=180] (-4.5,-1.05);
    \draw (-4.5,-1.05) to [out=0,in=180] (-3.9,-0.85);
    \draw (-3.9,-0.85) to [out=0,in=180] (-3.3,-1.05);
    \draw (-3.3,-1.05) to [out=0,in=0] (-3.3,-0.25);
    \draw (-3.3,-0.25) to [out=180,in=0] (-3.9,-0.45);
    \draw (-3.9,-0.45) to [out=180,in=0] (-4.5,-0.25);
    \draw  (-3.9,-0.65) circle (0.07);
    \draw [rounded corners] (-4.55,-2) -- (-4.55,-3) -- (-3.15,-3) -- (-3.15,-2) -- (-3.45,-2) -- (-3.45,-2.7) -- (-4.25,-2.7) -- (-4.25,-2) -- cycle;
    \draw [rounded corners] (-2.3,0) -- (-2.3,-1) -- (-2,-1) -- (-2,-0.3) -- (-1.3,-0.3) -- (-1.3,-1) -- (-1,-1) -- (-1,0) -- cycle;
    \draw (-2.15,-1.9) to [out=225,in=135] (-2.15,-2.9) to [out=-45,in=-60] (-1.85,-2.8) to [out=120,in=225] (-1.85,-2.1) to [out=45,in=45] (-2.15,-1.9);
    \draw (-1.15,-1.9) to [out=-45,in=45] (-1.15,-2.9) to [out=225,in=240] (-1.45,-2.8) to [out=60,in=-60] (-1.45,-2.1) to [out=135,in=135] (-1.15,-1.9);
    \draw [<->,thick] (-3.05,-1.8) -- (-2.45,-1.2);
    \draw  (-6.9,4.3) circle (0.1);
    \draw  (-6.3,4.3) circle (0.1);
    \draw  (-6.9,3.8) circle (0.1);
    \draw  (-6.3,3.8) circle (0.1);
    \draw  (-4.9,4.3) circle (0.1);
    \draw  (-4.4,4.3) circle (0.1);
    \draw  (-4.65,3.8) ellipse (0.4 and 0.1);

    \draw  (-6.6,2.5) ellipse (0.4 and 0.1);
    \draw  (-6.8,2) circle (0.1);
    \draw  (-6.4,2) circle (0.1);
    \draw  (-4.6,2.5) ellipse (0.4 and 0.1);
    \draw  (-4.6,2.1) ellipse (0.4 and 0.1);

    \draw (0.75, -1.5) -- (4.75,-1.5);
    \draw (2.75, -3.5) -- (2.75,0.5);
    \draw  (1,-2.4) ellipse (0.08 and 0.25);
    \draw  (2.3,-2.65) arc (-90:90:0.08 and 0.25);
    \draw [dashed] (2.3,-2.15) arc (90:270:0.08 and 0.25);
    \draw (1,-2.15) to [out=25,in=155] (2.3,-2.15);
    \draw (1,-2.65) to [out=-25,in=205] (2.3,-2.65);
    \draw  (1.3,-2.05) arc (90:-90:0.1 and 0.35);
    \draw  [dashed] (1.3,-2.05) arc (90:270:0.1 and 0.35);
    \draw  (1.85,-2.4) arc (0:-180:0.2 and 0.1); 
    \draw  (1.8,-2.46) arc (0:180:0.15 and 0.07); 
    \draw  (3.2,-0.6) ellipse (0.08 and 0.25); 
    \draw (3.2,-0.35) to [out=25,in=155] (4.5,-0.35); 
    \draw (4.5,-0.35) arc (90:-90:0.08 and 0.25); 
    \draw [dashed] (4.5,-0.35) arc (90:270:0.08 and 0.25); 
    \draw (3.2,-0.85) to [out=-25,in=205] (4.5,-0.85); 
    \draw  (4.06,-0.6) arc (0:-180:0.2 and 0.1); 
    \draw  (4.01,-0.66) arc (0:180:0.15 and 0.07); 
    \draw (4.22,-0.25) arc (90:-90:0.1 and 0.35); 
    \draw [dashed] (4.22,-0.25) arc (90:270:0.1 and 0.35); 
    \draw  (3.2,-2.4) ellipse (0.08 and 0.25); 
    \draw (3.2,-2.15) to [out=25,in=155] (4.5, -2.15); 
   \draw (4.5,-2.15) arc (90:-90:0.08 and 0.25); 
    \draw [dashed] (4.5,-2.15) arc (90:270:0.08 and 0.25); 
    \draw (3.2,-2.65) to [out=-25,in=205] (4.5,-2.65); 
    \draw  (4.06,-2.4) arc (0:-180:0.2 and 0.1); 
    \draw  (4.01,-2.46) arc (0:180:0.15 and 0.07); 
    \draw (3.85, -2) arc (90:-90:0.07 and 0.195); 
    \draw [dashed] (3.85, -2) arc (90:270:0.07 and 0.195); 
    \draw (3.85, -2.5) arc (90:-90:0.05 and 0.155); 
    \draw [dashed] (3.85, -2.5) arc (90:270:0.05 and 0.155); 
    \draw  (1,-0.6) ellipse (0.08 and 0.25); 
    \draw (1,-0.35) to [out=25,in=155] (2.3,-0.35); 
    \draw (2.3,-0.35) arc (90:-90:0.08 and 0.25); 
    \draw [dashed] (2.3,-0.35) arc (90:270:0.08 and 0.25); 
    \draw (1,-0.85) to [out=-25,in=205] (2.3,-0.85); 
    \draw  (1.86,-0.6) arc (0:-180:0.2 and 0.1); 
    \draw  (1.81,-0.66) arc (0:180:0.15 and 0.07); 

    \draw  (1.66,-0.63) ellipse (0.3 and 0.2);
  \draw   [dashed](1.26, -0.25) arc (90:270:0.1 and 0.35);
     \draw (1.26, -0.25) .. controls (2.35,-0.3) and (2.35,-0.95) .. (1.26, -0.95);
    \draw [<->,thick] (2.5,-1.25) to (3,-1.75);

    \node at (0, -4) {Arrows indicate sccs that must intersect};
  \end{tikzpicture}
  \caption{} \label{fig:resolved}
\end{figure}

Most of $M \subset \Sigma$ runs along arcs in the graphic, representing saddle walls for which one side is labelled $A$ and the other is labelled $B$.  A neighborhood of the corresponding saddle is a $3$-punctured sphere and crossing the saddle wall splits one curve into two:  $c \ra c^\pr \fakecoprod c^{\pr \pr}$.  

We start with a simple lemma.
\vspace{1em}

\begin{lemma}[two of three] \label{lemma:2of3}
  If crossing a saddle wall transforms $c \ra c^\pr \fakecoprod c^{\pr \pr}$ and if two of the three curves compress in $A$ (resp $B$) then so does the third.
\end{lemma}

\begin{proof}
Consider capping off two boundaries of the 3-punctured sphere.
\end{proof}

Of course the third curve may be inessential in $T$, so crossing from a region labelled $A$ to one labelled $B$ may represent two curves bounding essential disks in $A$ fusing at the saddle into a curve inessential in $T$, leaving other unrelated curves bounding disks in $B$. (In this case the two fused curves at $A$ are isotopic in $T$.) Or passing through could represent a curve that bounds a disk in $A$ fusing with a curve that is essential in both $A$ and $B$ to create a curve bounding a disk in $B$, or a similar fission.  In each case there is a unique essential intersection curve $a$ bounding a disk in $A$ associated with the saddle wall and it is disjoint from every essential intersection curve bounding a disk in $B$.

To expedite the discussion we introduce: 
\vspace{0em}

\begin{defin} \label{def:conventions}
\begin{itemize}  Some conventions:
  \item We call a scc $a$ (resp $b$) if it is essential in $T$ and compresses in $A$ (resp $B$). 
  \item An edge-path in $2C(T)$ in which each vertex is of type $(a, b)$ (as just defined) is an \underln{admissible} path.
  \item Near a saddle wall crossing (swc) point $p$ we call a scc of type $a$ or $b$ \underln{local} if it is a scc of one of four resolutions of $\xi_P$.  Otherwise we call it \underln{far}.
  \end{itemize}
\end{defin}

With these conventions, we have just observed that any saddle wall that is part of $M$ is associated to a `cloud' of vertices in $2C(T)$ of the form $(a, b)$, with $a$ uniquely defined, and $b$ any curve of that type coming from the  adjacent graphic region labelled $B$.  Any two vertices in such a cloud are connected via the relation 
$$(a, b) \ra (a, b^\pr) \: \textrm{since $a$, $b$, $b^\pr$ are all pairwise disjoint.}$$

Recalling our artistic rendering convention for defining the 1-manifold $M$, we now consider how these clouds of vertices in $2C(T)$ are related as we pass from one saddle wall to another, either through or around a corner of a vertex in the graphic representing a saddle wall crossing.

\vspace{1em}
 \begin{lemma}
   For all cases (i)
   $\arraycolsep = 2.0pt
   \begin{array}{c|c}
     A & A \\
     \hline
     B & B
   \end{array}$
   , (ii)
   $\arraycolsep = 2.0pt
   \begin{array}{c|c}
     A & B \\
     \hline
     B & A
   \end{array}$
   , (iii)
   $\arraycolsep = 2.0pt
   \begin{array}{c|c}
     A & A \\
     \hline
     A & B
   \end{array}$
   , and (iv)
   $\arraycolsep = 2.0pt
   \begin{array}{c|c}
     A & B \\
     \hline
     B & B
   \end{array}$
there is an admissible path in $2C(T)$ from some vertex in one cloud to some vertex in the other.  
 \end{lemma}


\begin{proof}
Since each pair of vertices in a cloud are connected by an edge in $2C(T)$, and such an edge is obviously an admissible path, it suffices to replace both occurences of `some' with `every' in its two occurences in the lemma.  
\medskip

   Case (i)
   $\arraycolsep = 2.0pt
   \begin{array}{c|c}
     A & A' \\
     \hline
     B & B'
   \end{array}$
   : If any $b$ (or any $b^\pr$) is far it will be unaffected by passing through any saddle wall, so pick $b = b^\pr =$ far. Moreover, passing through the saddle wall between the regions $A$ and $A'$ will change the $a$ curve to a disjoint $a'$.  Then $(a, b) \ra (a^\pr, b) = (a^\pr, b^\pr)$ is an admissible path from one cloud to the other.

   If $b$ and $b^\pr$ are both local there is a subtlety. In all cases at most one diagonal pair of quadrants contain intersecting curves (see Figure \ref{fig:resolved}); call this a \underln{dangerous diagonal}. But the appropriate choice of a two-step process avoids the dangerous diagonal, so all curves are disjoint and the path is admissible:
   \begin{tikzpicture}
    [baseline={([yshift=9pt]current bounding box.south)}]
     \node at (0,1) {$(a, b)$};
     \draw [->] (0.6, 1) -- (1.6, 1.4);
     \draw [->] (0.6, 1) -- (1.6, 0.6);
     \node [right] at (1.6, 1.4) {$(a, b^\pr)$};
     \node [right] at (1.6, 0.6) {$(a^\pr, b)$};
     \node [right] at (1.92, 1) {or};
     \draw [->] (2.8, 1.4) -- (3.8, 1.1);
     \draw [->] (2.8, 0.6) -- (3.8, 0.9);
     \node [right] at (3.8, 1) {$(a^\pr, b^\pr)$};
   \end{tikzpicture}
  
   Case (ii)
   \begin{tikzpicture}
     [baseline={([yshift=9pt]current bounding box.south)}]
     \draw [rounded corners] (0, 0.4) -- (0.4, 0.4) -- (0.4, 0);
     \draw [rounded corners] (0.45, 0.85) -- (0.45, 0.45) -- (0.85, 0.45);
     \node [above right] at (-0.1, -0.1) {B};
     \node [above right] at (-0.1, 0.4) {A};
     \node [above right] at (0.4, 0.4) {B};
     \node [above right] at (0.38, -0.1) {A};
   \end{tikzpicture}
   : It suffices to consider one arc of $M$. Regardless of whether $b$ is local or far leave it fixed. The only problematic possibility for a transition $(a, b) \ra (a^\pr, b)$ is if $a$ and $a^\pr$ are intersecting curves from a dangerous diagonal. In this context case (5) cannot occur since there $a$ and $a^\pr$ would have intersection number = 1. But where it is possible, cases (2), (3), and (4), the topology is the same: $a$ and $a^\pr$ are crossing non-boundary-parallel sccs on a 4-punctured sphere:

   \begin{figure}[ht] 
     \centering
     \begin{tikzpicture}
       \draw (0, 2) ellipse (0.25cm and 0.5cm);
       \draw (2, 0) ellipse (0.5cm and 0.25cm);
       \draw (2, 4) ellipse (0.5cm and 0.25cm);
       \draw (4, 2) ellipse (0.25cm and 0.5cm);

       \draw (0, 2.5) to [out = 20, in = 250] (1.5, 4);
       \draw (1.5, 0) to [out = 110, in = -20] (0, 1.5);
       \draw (2.5, 0) to [out = 70, in = 200] (4, 1.5);
       \draw (4, 2.5) to [out = 160, in = 290] (2.5, 4);

       \draw [rotate = 45] (1.32,0) arc (180: 360: 1.5cm and 0.5cm);
       \draw [dashed, rotate = 45] (1.32,0) arc (180: 0: 1.5cm and 0.5cm);
       \path [pattern = vertical lines, rotate around = {45:(2,2)}] (2, 2) ellipse (1.5 cm and 0.5cm);

       \draw [rotate around = {-45:(0.94,3.06)}] (0.94,3.06) arc (180: 360: 1.5 cm and 0.5 cm);
       \draw [dashed, rotate around = {-45:(0.94,3.06)}] (0.94,3.06) arc (180: 0: 1.5 cm and 0.5 cm);
       \path [pattern = north east lines, rotate around = {-45:(2,2)}] (2, 2) ellipse (1.5 cm and 0.5cm);
     \end{tikzpicture}
     \caption{} \label{fig:2crossings}
   \end{figure}

   But an innermost circle argument reduces the geometry to precisely the case pictured in Figure \ref{fig:2crossings}, where we see that all four boundary curves also compress in $A$. By Lemma \ref{lemma:2of3} at most two of these boundary curves are inessential in $T$; let $\de$ be a boundary curve that is essential in $T$.  
   Then $(a, b) \ra (\de, b) \ra (a^\pr, b)$ is an admissible path from one cloud to the other. (Note that the intermediate vertex $(\de, b)$ is not in either cloud.)  Call this the \underln{lantern trick}.

   Case (iii)
   \begin{tikzpicture}
     [baseline={([yshift=9pt]current bounding box.south)}]
     \draw [rounded corners] (0.45, 0) -- (0.45, 0.45) -- (0.85, 0.45);
     \node [above right] at (-0.1, -0.1) {A};
     \node [above right] at (-0.1, 0.4) {A\textsuperscript{$\pr$}};
     \node [above right] at (0.37, 0.4) {A\textsuperscript{$\pr \pr$}};
     \node [above right] at (0.4, -0.1) {B};
   \end{tikzpicture}
   : This is similar to case (ii). Regardless of whether $b$ is local or far, leave it fixed. The only problematic case in passing from $(a, b)$ to $(a^{\pr \pr}, b)$ is if $a$ and $a^{\pr \pr}$ are intersecting sccs from a dangerous diagonal. As in case (ii), case (5) cannot occur. Cases (2), (3), and (4) are also handled as in case (iii) by the lantern trick.

   Case (iv)
   \begin{tikzpicture}
     [baseline={([yshift=9pt]current bounding box.south)}]
     \draw [rounded corners] (0, 0.43) -- (0.42, 0.43) -- (0.42, 0.85);
     \node [above right] at (-0.07, -0.1) {B};
     \node [above right] at (-0.1, 0.4) {A};
     \node [above right] at (0.37, 0.4) {B\textsuperscript{$\pr \pr$}};
     \node [above right] at (0.37, -0.1) {B\textsuperscript{$\pr$}};
   \end{tikzpicture}
   : If any $b$, $b^\pr$, $b^{\pr \pr}$ is far, select it and keep it constant, and keep $a$ constant as well.  Then the corresponding vertex in $2C(T)$ lies in both clouds
   
   Now assume all $b$, $b^\pr$, $b^{\pr \pr}$ are local.

   Case (1): If $a, b, b', b''$ are all disjoint, then $(a, b) \ra (a, b^{\pr \pr})$ is an admissible path.  

   Case (2), (3), (4): The only non-disjoint situation is when $b$ and $b^{\pr \pr}$ are intersecting curves from a dangerous diagonal pair. Again use the lantern trick.

   Case (5): This case cannot occur:  By intersection number, the three $B$ regions cannot contain a dangerous diagonal pair, so $A$ must be one of the dangerous diagonal pair of quadrants. Then the other quadrant in that pair must be $B^\pr$.  The other two quadrants then are represented by the boundary curves in the twice-punctured torus shown in Figure \ref{fig:resolved}, one compressible in $A$ and the other in $B$.  In order to avoid a label $A$ in one of these quadrants, the former curve must be inessential in $T$ so the twice-punctured torus is actually only once-punctured.  But in that case, the label $B$ (or $B''$) for this quadrant implies that there is a (far) $b$ curve, contradicting assumption.   
   \end{proof}
%

We now complete the proof of Theorem \ref{thm:cwr} as described in the plan above.  Pick an essential component of the $1$-manifold $M$ in $\Sigma$ and a saddle wall edge in it.  Beginning at that edge and moving around the component of $M$, construct an admissible path through $2C(T)$ as described above, until one returns to the original edge.  Since the component is homotopic in $I \times I \times S^1$, rel its beginning and ending point, to a simple loop $\{pt\} \times S^1$ which represents $\tau$, the cwr given by the admissible path also represents $\tau$. \end{proof}

\section{The genus $3$ case} \label{sect:genus3}

We are now in a position to prove the Powell Conjecture for genus $3$ splittings of $S^3$, by verifying Conjecture \ref{conj:cwr} in this case.

We need a bit of terminology.  Suppose $H$ is a genus $3$ handlebody and $D \subset H$ is a separating disk.  Then $D$ divides $H$ into a solid torus $H_D$ and a genus $2$ handlebody; the meridian disk $D'$ of $H_D$ is well-defined up to proper isotopy in $H$ and is non-separating in $H$; call $D'$ the {\em surrogate} of $D$ in $H$.  

\begin{lemma}  \label{lemma:step1}Suppose $S^3 = A \cup_T B$ is a genus $3$ splitting, and disks $a \subset A, b \subset B$ are a weakly reducing pair.  Then at least one of the following $4$ disks is primitive:  
\begin{itemize} 
\item $a$, if $a$ is non-separating in $A$,
\item the surrogate of $a$ if $a$ is separating in $A$
\item $b$, if $b$ is non-separating in $B$
\item the surrogate of $b$ if $b$ is separating in $B$
\end{itemize}
\end{lemma}

\begin{proof}  
{\bf Case 1:  $a$ is separating.}  One possibility in this case is that $\bdd b$ lies on the boundary of the solid torus $A_a \subset A$. In that case, if $b$ is non-separating, then $\bdd b$ must be a longitude of $A_a$, verifying that the surrogate of $a$ (and incidentally also $b$) is primitive.  If $b$ is separating, then $\bdd b$ must be parallel to $\bdd a$ in $T$, so together $a$ and $b$ constitute a reducing sphere for $T$, cutting off a genus $1$ Heegaard splitting whose splitting surface is $\bdd A_a$.  It follows again that the surrogate of $a$ (i. e. the meridian of $A_a$) is primitive (and incidentally that the surrogate of $b$ is also primitive).

The other possibility is that $\bdd b$ is essential in the punctured genus $2$ component of $T - \bdd a$.  This implies that the surrogate of $b$ (if $b$ is separating) is disjoint from the surrogate of $b$, so we may as well jump immediately to Case 2.

{\bf Case 2:  Both $a$ and $b$ are non-separating.}  Let $T'' \subset S^3$ be the torus obtained by simultaneously compressing $T$ along $a$ and $b$.  $T''$ bounds a solid torus $W$ on one side or the other, say $W$ lies on the side in which $b$ lies.  The surface $T'$ obtained from $T$ by compressing only along $a$, when pushed into $W$, describes a genus $2$ Heegaard splitting $W = A' \cup_{T'} B'$ in which $\bdd W = \bdd_- B'$ and $b$ is the unique boundary-reducing disk for the compression body $B'$.  It follows that $b$ is primitive in $T'$, hence in $T$.
\end{proof}

\bigskip

Following Lemma \ref{lemma:step1} and Corollary \ref{cor:orthog2} there is an obvious candidate for assigning a Powell equivalence class of homeomorphisms to the pair $(a, b)$:  if $a$ (or its surrogate) is primitive, assign the class of homeomorphisms that carries that disk to the disk $a_1$ in the standard splitting $A \cup_{T_3} B$.  If $b$ (or its surrogate) is primitive, assign the class of homeomorphisms that carries that disk to $b_3$ in the standard splitting.  It is only necessary to check that if two of the disks listed in Lemma \ref{lemma:step1} are primitive, the assignments we have made coincide.  So suppose both $a$ (or its surrogate) and $b$ (or its surrogate) are primitive.

If both $a$ and $b$ are non-separating (so surrogates are not involved) and therefore primitive, a straightforward outermost arc argument will find a disk $b_{\perp} \subset B$ that is orthogonal to $a$ and disjoint from $b$.  Similarly, one can then find a disk $a_{\perp} \subset A$ that is dual to $b$ and is disjoint from both $a$ and $b_{\perp}$.  There is then an obvious homeomorphism $(S^3, T) \to (S^3, T_3)$ (well-defined up to the $2$-strand braid group of the torus) that takes the four disks $a, b_{\perp}, b, a_{\perp}$ to, respectively $a_1, b_1, b_3, a_3$. The Powell class of this homeomorphism then coincides both with that determined by $a$ and that determined by $b$.  

There remains the possibility that one or both primitive disks are surrogates of $a$ or $b$ and intersect, in which case they are orthogonal.  There is then a homeomorphism $(S^3, T) \to (S^3, T_3)$ that carries one disk to $a_1$ and the other one to $b_1$.  But a further Powell move (an exchange) carries the pair to $a_3$ and $b_3$.  The first homeomorphism is in the Powell equivalence class determined by $a$ or its conjugate and the second by $b$ or its conjugate.  Since there is a Powell move taking one to the other, the equivalence classes coincide.  

The second and final step is to show that any two vertices in $2C(T)$ that are connected by an edge determine the same Powell equivalence class.  Given the simple rule above for assigning a Powell equivalence class, this will follow immediately from the following lemma.

\begin{lemma} \label{lemma:step2} For the genus $3$ Heegaard splitting $S^3 = A \cup_T B$, suppose $a \subset A$, $b_1, b_2 \subset B$ are three pairwise disjoint essential disks .  Then either
\begin{itemize}
\item $a$ is primitive
\item $a$ is separating and its surrogate is primitive
\item $b_1$ and $b_2$ are parallel in $B$ or
\item $b_1$ is the surrogate of $b_2$ or vice versa.
\end{itemize}
The symmetric statement is true for three pairwise disjoint essential disks $(a_1, a_2, b)$
\end{lemma}  

\begin{proof}  {\bf Case 1:  $a$ is separating.}  If the boundary of either $b_i$ lies on the same side of $a$ as $a'$ then the disk is either orthogonal to $a'$ or the disk together with $a$ is a reducing sphere cutting off a genus $1$ summand with $a'$ as meridian.  In any case, $a'$ is primitive and we are done.  So we henceforth assume that the boundary of both $b_i$ lie on the genus $2$ side of $T - \bdd a$, and neither is parallel to $\bdd a$.

If, together, $b_1$ and $b_2$ separate $T$ then the component containing $\bdd a$ is either a torus in which $a'$ is a meridian, and so is primitive, or a genus $2$ surface.  In the latter case, the other component must be attached either by a single disk that cuts off a torus with the other disk as meridian (so one of $b_1, b_2$ is surrogate to the other) or attached by two disks, one copy of each $b_i$, that cut off a sphere, in which case $b_1$ and $b_2$ are parallel in $B$.  

{\bf Case 2: $a$ is non-separating.}  One of the components of $B - (b_1 \cup b_2)$ contains $\bdd a$ and has connected boundary; attaching a $2$-handle to that component along $a$ creates a $3$-manifold $N$ that will still be connected.  

Also $\bdd N$ is still connected, for if it were disconnected, with boundary components $\bdd_1 N, \bdd_2 N$, each containing a copy of $a$, then there would be a path through that component of $B - (b_1 \cup b_2)$ from $\bdd_1 N$ to $\bdd_2 N$.  But there is already a path from $\bdd_1 N$ to $\bdd_2 N$ in $T - a$, since $a$ is non-separating.  Together the paths would give a circle that intersects the surface $\bdd_1 N$ in a single point, which is impossible in $S^3$.  

So let $\gamma \subset \bdd N$ be a path from one copy of $a$ in $\bdd N$ to the other.  Tubing the copies of $a$ together along the boundary of a neighborhood of $\gamma$ gives a separating disk $a^+$ in $A$, disjoint from the $b_i$, and $a$ is the surrogate for $a^+$.  Now apply case 1 to the three disks $a^+, b_1, b_2$.  
\end{proof}

\begin{cor}  The Powell Conjecture is true for genus $3$ splittings.
\end{cor}

\begin{proof}  Suppose two vertices in $2C(T)$ are connected by an edge.  With no loss of generality, the two vertices are $(a, b_1)$ and $(a, b_2)$ with the three disks $a, b_1, b_2$ pairwise disjoint.  In each case given by Lemma \ref{lemma:step2} the procedure described for assigning a Powell equivalence class assigns the same equivalence class to both cases.
\end{proof}

\section{Concluding remarks}

The methods here seem to hold great promise for the general case, genus $g \geq 4$.  But there are many difficulties:  Presented with a pair of weakly reducing disks $a \subset A, b \subset B$, the methods of Casson-Gordon (\cite{CG}) describe how to proceed to exhibit a complete non-separating weakly reducing collection of disks $a_1, ..., a_{g_1}, b_{g_1 + 1}, ..., b_g$ for $T$.  But many choices are involved in creating such a collection, and it is not clear that different choices will lead, via Theorem \ref{thm:cindep}, to the same Powell equivalence class of homeomorophisms to $T_g$, even under strong inductive assumptions.  

To begin with, the first step in the Casson-Gordon recipe, following the choice of $a, b$, is to use the pair to maximally weakly reduce the splitting.  But this involves choices: for example, one could choose first to expand $a$ to a maximal collection of compressing disks in $A$ that are disjoint from $b$.  (Or one could do the reverse!)  The resulting compression body lying in $A$ may itself be well-defined, via the arguments of \cite{Bo}, but there are many possible families of compressing disks for the compression body, and, to apply Theorem \ref{thm:cindep}, a specific collection of compressing disks in $A$ is required.   

The next step in \cite{CG} is to maximally compress the surface $T_{A_b}$ just defined towards the side that contains $B$.  The same difficulty arises: the actual compression disks into the $B$ side that one uses involves a choice.  On the positive side, the resulting surface $T_{BA_b}$, which is well-defined, can be shown to bound a handlebody $M_{BA_b}$ in which $T_{A_b}$ is a Heegaard surface, providing perhaps an inductive foothold.  Furthermore,  Powell-like moves (bubble-moves, eyeglass twists, etc) in the Heegaard surface of $T_{A_b} \subset M_{BA_b}$ can be 'shadowed' by similar moves in the original splitting surface $T$.  But again choice is involved in how to shadow the moves.  Also, even once a primitive disk is chosen for the splitting $T_{A_b} \subset M_{BA_b}$ (inductively we might assume that this choice is unique up to Powell-like moves), such a disk does not pick out a {\em unique} primitive disk in our original surface $T$.  

All along one could hope to show that different choices made lead to Powell equivalent trivializations, but we have failed to find such arguments.

\medskip

Perhaps a good intermediate question is this:

\begin{conj}  Any eyeglass twist on $T_g$ is a Powell move. \label{conj:eyeglass}
\end{conj}

Of course any eyeglass twist is a Goeritz element, so this conjecture is weaker than the Powell Conjecture.  An affirmative solution would bode well for an affirmative solution of the Powell Conjecture as well.  A weaker but perhaps still challenging conjecture is this:

\begin{conj}  Any eyeglass twist on $T_g$ is Goeritz conjugate to a Powell move. 
\end{conj}

That is, given an eyeglass twist $\rho: (S^3, T_g) \to (S^3, T_g)$ there is a homeomorphism $h: (S^3, T_g) \to (S^3, T_g)$ so that $h^{-1}\rho h: (S^3, T_g) \to (S^3, T_g)$ is a Powell move.  

Observe the following cautionary example about eyeglass twists:  In the argument of Section \ref{sub:annular} a crucial weakly reducing pair that may well appear can be viewed as follows:  One of the pair, $b \subset B$ say, is an innermost disk of $B \cap S$ for some level sphere $S$; $b$ lies inside another disk $D$ that is also bounded by a curve of $S \cap T$, a curve to which $b$ is tangent.  Finally, the complement of $b$ in $D$  lies entirely in $A$ and becomes the disk $a$ when the point of tangency is resolved.  See Figure \ref{fig:tangency}.

     \begin{figure}[ht!]
\labellist
\small\hair 2pt
\pinlabel  $b$ at 100 75
\pinlabel  $a$ at 135 140
\pinlabel  $\bdd D$ at 45 110
\endlabellist
    \centering
    \includegraphics[scale=0.75]{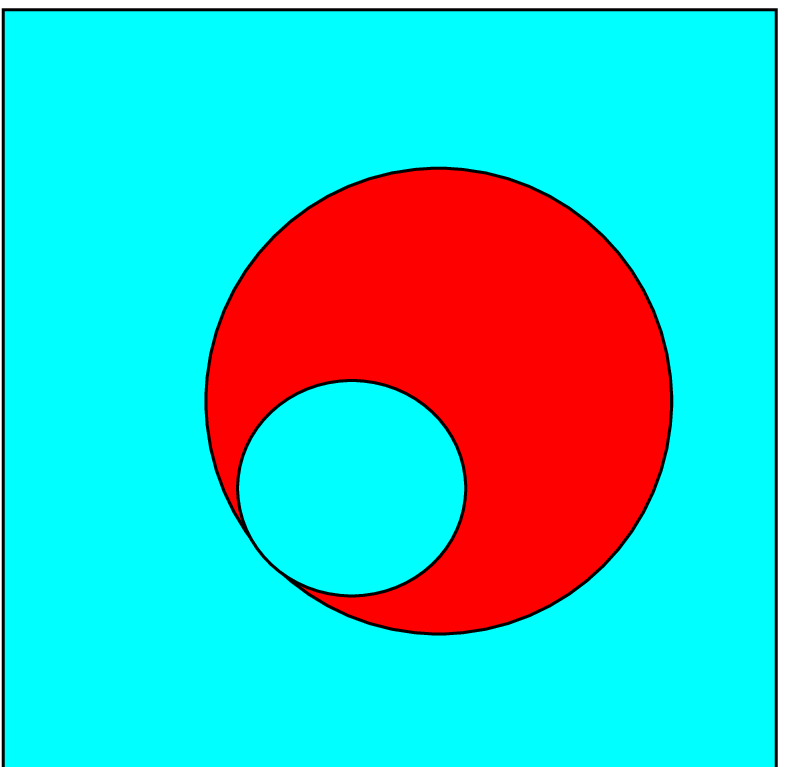}
\caption{} 
 \label{fig:tangency}
    \end{figure}

Now imagine rolling the disk $b$ around the inside of $D$ so that the point of tangency circumnavigates $\bdd D$.  It is easy to see that this is a Goeritz move and only a bit harder to see that it is an eyeglass twist of $b$ around $a$.  Yet, the methods of Section \ref{sub:annular} would provide the  same pair of disks $a, b$ throughout the Goeritz move.  This suggests that, as it stands, Section \ref{sub:annular} is not in itself powerful enough to resolve even the weaker Conjecture \ref{conj:eyeglass}.  Other tools may be needed to solve the problem for at least some types of eyeglass twists.

In that spirit, we conclude with one rather limited class of eyeglass twists which we are able to show are Powell moves.  These are described using the notation that precedes Lemma \ref{lemma:eyeglass4}.  
 
\begin{defin} \label{defin:eyeglass5} An eyeglass $\eta \subset T$ satisfying the following conditions will be called a {\em short} eyeglass.  
\begin{itemize}
\item There is a complete non-separating pair of weakly reducing disk collections $\{a_1, ..., a_{g_1}\}$ and $\{b_{g_1+1}, ..., b_g\}$ for $T \subset S^3$. 
\item There is a simple closed curve $c \subset T$ that separates the two sets $\{\bdd a_1, ..., \bdd a_{g_1}\}$ and $\{\bdd b_{g_1+1}, ..., \bdd b_g\}$ in $T$.
\item The lens $a \subset A$ of $\eta$ has $\bdd a \subset T_A$.
\item  The lens $b \subset B$ of $\eta$ has $\bdd b \subset T_B$.
\item The interior of the bridge $v \subset \eta$ intersects at least one of $T_A$ or $T_B$ entirely in $P_A$ or $P_B$.  
\end{itemize}
\end{defin}

\begin{prop} \label{prop:eyeglass5} Suppose $\eta \subset T$ is a short eyeglass as above, with $c \subset T$ the separating curve. Then an eyeglass twist along $\eta$ does not change the Powell equivalence class of any homeomorphism $h:(S^3, T, c) \to (S^3, T_g, c_{g_1})$.
\end{prop}

The central difference between this lemma and Lemma \ref{lemma:eyeglass4} is that we no longer require the bridge $v$ to intersect $c$ in exactly one point; on the other hand, we do require that the bridge of the eyeglass be disjoint either from $\{\bdd a_1, ..., \bdd a_{g_1}\}$ or from $\{\bdd b_{g_1+1}, ..., \bdd b_g\}$

\begin{proof}    The proof is by induction on $|v \cap c|$; the case $|v \cap c| = 1$ is Lemma \ref{lemma:eyeglass4}.  

Suppose with no loss of generality that $v \cap T_A \subset P_A$ and that $|v \cap c| > 1$.  Consider the arcs of $v \cap P_A$ that are not incident to $\bdd a$.  Each such divides $P_A$ into two sub-planar surfaces, one of which cuts off a subarc of $c$ intersecting $v$ in an even number of points.  By picking an outermost such arc we may find a properly embedded segment $v_A \subset v$ that cuts off a planar surface whose interior is disjoint from $v$.  The union of $v_A$ and the subarc $c_A \subset c$ it cuts off is a simple closed curve $a_c$ which bounds a disk in $A$.  Create an eyeglass $\eta'$, one of whose lenses is $a_c$, the other is $b \subset B$ and whose bridge intersects $c$ in at least two fewer points.  By induction, a twist around $\eta'$ does not change the Powell equivalence class of $h$, and it vacates entirely everything in $P_A$ that lies between $v_A$ and $c_A$, putting it on the other side of $v_A$. Then $v_A$ can be isotoped across $c$, reducing $|v \cap c|$.
\end{proof}

\appendix

\section{Refining the Rieck argument} \label{appendix}

We begin with a well-known result.  Suppose $X$ is a compact topological space and $f: X \to \Rrr$ is continuous.  Then $f$ has a unique minimal value $r \in \Rrr$.  Now let $Z$ be a topological space, $\Rrr^X$ denote the space of continuous functions from $X$ to $\Rrr$ with the compact-open topology and consider a function $f: Z \to \Rrr^X$.  Denote the image of any $z \in Z$ by $f_z: X \to \Rrr$.  Then $f$ gives rise to a function $r: Z \to \Rrr$ by defining $r(z) = min\{f_z(x) | x \in X\}$.  

\begin{prop}  If $f$ is continuous then so is $r$.
\end{prop}

\begin{proof} The standard proof is this:  Since $X$ is compact, the compact-open topology coincides with the topology of uniform convergence, in which the conclusion is obvious.  There is also a proof directly from the definition of compact-open topology that is analogous to (but much easier than) the proof of Proposition \ref{prop:H1graph} below.
\end{proof}

Notice that the function $r$ could equivalently (but more obscurely) be defined this way:  $r(z)$ is the least $r \in \Rrr$ so that the image of $H_0(f_z^{-1}(- \infty, r]) \to H_0(X)$ is non-trivial.  Viewed in this way, the proposition generalizes.  In particular, we will be interested in an analogous statement for $H_1$.  

For this version, assume also that $H_1(X)$ is non-trivial.  Then given a continuous
 function $f: X \to \Rrr$ and a value $y \in \Rrr$, let $G_y \subset H_1(X)$ be the image under the inclusion-induced map of $H_1(f^{-1}(-\infty,y]) \to H_1(X)$.  There are obvious properties:  
\begin{itemize}
\item for $y$ below the minimum of $f(X)$, $G_y$ is trivial,  
\item $y \leq y' \implies G_y \subset G_{y'}$,  and
\item  for $y$ above the maximum of $f(X)$, $G_y = H_1(X) \neq 0$. 
\end{itemize}  
So it makes sense in this situation to define a value \[r = \inf\{y \in \Rrr| G_y \neq 0\}.\footnote{Notice that $G_r$ may or may not be trivial.  For an example where $G_r$ is trivial, let $X$ be the closure of an $\epsilon$-neighborhood of the Polish circle $C$ in the plane, and $f(x)$ be the distance from $x$ to $C$.  Then $r = 0$ and $G_r = H_1(C) = 0.$}\]  With this definition then, in analogy to the discussion above, a function $f: Z \to \Rrr^X$ gives rise to a function $r: Z \to \Rrr$ by setting $r(z)$ to be the value $r$ just defined for the function $f_z$.  That is, \[r(z) = \inf\{y \in \Rrr|  image(H_1(f_z^{-1}(-\infty,y]) \to H_1(X)) \neq 0\}\]

\begin{prop}  \label{prop:H1graph}
Since $f$ is continuous, so is $r$.
\end{prop}

\begin{proof}  The proof (not the underlying mathematics) is made a bit more complicated because we we do not know when $G_r$ is trivial.

Pick any $z_0 \in Z$ and any $\eee > 0$.  Let $r = r(z_0)$.  

\medskip
{\bf Claim 1:}  There is an open neighborhood $N$ of $z_0$ so that $z \in N \implies r(z) < r + 2\eee$.   

{\em Proof of Claim 1:} First observe that $f_{z_0}^{-1}(- \infty, r + \eee]$ is closed in $X$, hence compact, so $\{z \in Z| f_z(f_{z_0}^{-1}(- \infty, r + \eee]) \subset (-\infty, r + 2\eee)\}$ is an open neighborhood $N$ of $z_0$.  Put another way, for $z \in N$, \[ f_{z_0}^{-1}(- \infty, r + \eee] \subset f_{z}^{-1}(- \infty, r + 2\eee).\]  We know that the image of $H_1(f_{z_0}^{-1}(- \infty, r + \eee] ) \to H_1(X)$ is non-trivial, and this factors through $H_1(f_{z}^{-1}(- \infty, r + 2\eee)) \to H_1(X)$, so the image of the latter is non-trivial.  Hence $r(z) < r + 2\eee$ as required.  
\medskip

{\bf Claim 2:}  There is an open neighborhood $M$ of $z_0$ so that $z \in M \implies r(z) > r - 2\eee$.   

{\em Proof of Claim 2:} Paralleling the argument in Claim 1, there is an open neighborhood $M$ of $z_0$ so that for $z \in M$, 
\[X - f_{z_0}^{-1}(- \infty, r - \eee) = f_{z_0}^{-1}[r - \eee, \infty) \subset f_{z}^{-1}(r- 2\eee, \infty) = X - f_{z}^{-1}(-\infty, r - 2\eee].\]  Hence 
\[f_{z}^{-1}(-\infty, r - 2\eee] \subset f_{z_0}^{-1}(- \infty, r - \eee)\]
By definition of $r$, $H_1(f_{z_0}^{-1}(- \infty, r - \eee)) \to H_1(X)$ has trivial image; it follows that $H_1(f_{z}^{-1}(- \infty, r- 2\eee]) \to H_1(X)$ has trivial image, so $r(z) > r - 2\eee$ as required.

Hence for $z \in M \cap N$, $|r(z) - r(z_0)| < 2\eee$, showing that $r$ is continuous at $z_0$.  
\end{proof}  

Now return to the  setting of Section \ref{sub:rieck}: $T$ is a genus $g \geq 2$ Heegaard surface in $S^3$ and $T_t, t \in [-1, 1]$ describes a sweep-out of $S^3 = A \cup_T B$ from a spine of $A$ to a spine of $B$. Similarly $S_s, s \in [-1, 1]$ describes the sweep-out of $S^3$ by $2$-spheres, from the south pole no the north pole of $S^3$.  The intersection patterns of the two surfaces give rise to a reduced graphic $\Gamma \subset [-1, 1] \times [-1, 1]$ for which a complementary region $R$ is labelled $E$ if and only if for each $(t, s) \in R$, $T_t \cap S_s$ contains a circle that is essential in $T_t$.   

\begin{prop} \label{prop:graph}
There is a function $r: [-1, 1] \to [-1, 1]$ whose graph $\{(t, r(t))\} \subset [-1, 1] \times [-1, 1]$ is transverse to $\Gamma$ and each region that the graph intersects is labelled $E$.
\end{prop}

\begin{proof} Apply Proposition \ref{prop:H1graph} to $X = T$ and $Z = [-1, 1]$ by taking for $f_t$ the height function $s$ restricted to $T_t \subset S^3$.  Denote by $T_{t, (-\infty s]}$ the subsurface of $T_t$ that lies in $\cup_{s' \leq s} S_{s'}.$, that is, in the part of $S^3$ that is south of the sphere $S_s$.     We then have a function $r:[-1, 1] \to [-1, 1]$ with the property that for values of $s < r(t)$, none of the first homology of $T_{t, (-\infty, s]}$ persists into all of $T_t$, but for any values of $s >  r(t)$ some does.  It follows that the topology of $T_{t, (-\infty s]}$ changes as $s$ rises through $r(t)$, so the graph of $r(t)$ must lie in the $1$-skeleton of the reduced graphic.  

We can say more:  For a generic value of $t$, the point $(t, r(t))$ will lie in an edge of the reduced graphic, so at that point there is a saddle tangency of $T_t$ with $S_{r(t)}$.  That is, $T_{t, (-\infty, r(t) + \eee]}$ is obtained from $T_{t,(-\infty, r(t) - \eee]}$ by attaching a $1$-handle.  

At a non-generic value of $t$, the point $(t, r(t))$ will be a vertex in the graphic.  In that case, one can get from the region just below $(t, r(t))$ to the region just above by passing through two edges in the graphic, so $T_{t, (-\infty, r(t) + \eee]}$ is obtained from $T_{t, (-\infty, r(t) - \eee]}$ by attaching two $1$-handles.  By definition of $r(t)$, $T_{t, (-\infty, r(t) - \eee]}$ corresponds to a subsurface of $T_t$ whose first homology is trivial in all of $T_t$.  In particular, $T_{t, (-\infty, r(t) + \eee]}$ contains no non-separating simple closed curves i. e., $T_{t, (-\infty, r(t) + \eee]}$ is planar.    

We now invoke an easy lemma:

\begin{lemma} \label{lemma:planar}

Suppose $P$ is a compact possibly disconnected planar surface and $P_+$ is the surface obtained from $P$ by attaching no more than three $1$-handles.  Then either genus $P_+ \leq 1$ or some boundary component of $P$ is non-separating in $P_+$.  
\end{lemma}

\begin{proof}  If any of the $1$-handles has its ends attached to different components of $P$, it is wasted:  The resulting surface $P'$ is still planar, and the new boundary component created is non-separating in $P_+$ if and only if at least one of the boundary components to which it is attached is non-separating in $P_+$.  So we could juat replace $P$ with $P'$ in the argument, but now with one fewer handle to attach.

So we may as well assume that all $1$-handles are attached to the same component of $P$ and so we can take $P$ to be connected.  In that case, if any $1$-handle has exactly one end on any component of $\bdd P$ then that component becomes non-separating, finishing the argument.  So we may as well assume that no $1$-handle has a single end on an original boundary component of $P$. 

Now imagine attaching the $1$-handles sequentially, say $h_1, h_2, h_3$.  The first handle $h_1$ must have both ends on the same component of $\bdd P$ so the result is still a planar surface $P'$, with exactly two components of $\bdd P'$ not appearing as components of $\bdd P$.  If, after attaching $h_2$, the surface is still planar, attaching $h_3$ can raise the genus to at most $1$, completing the proof.  So we need to have the attachment of $h_2$ raise the genus, so it must connect the two new boundary components of $P'$.  The resulting surface $P''$ now has genus $1$ and only a single boundary component not among the components of $\bdd P$.  It follows that both ends of $h_3$ must lie on the same component of $\bdd P''$, either the new one or one of the original components of $\bdd P$.  In any case, attaching $h_3$ does not raise the genus any further, completing the proof.
\end{proof}

Following Lemma \ref{lemma:planar}, if every boundary component of $T_{t, (-\infty, r(t) + \eee]}$ is inessential in $T_t$ then $genus(T_t) \leq 1.$  So under the hypothesis that $genus(T) \geq 2$ the function $r(t) + \eee$ for small enough $\eee$ will be the function we seek: its graph in $[-1, 1] \times [-1, 1]$ will lie entirely in the regions labelled $E$.   
\end{proof}

The argument is readily enlarged to include the third parameter $\te$.  That is,

\begin{prop} \label{prop:graphsum} In the parameter space $[-1, 1] \times [-1, 1]\times S^1$ as described in Section \ref{sub:rieck}, there is a function $r: I \times S^1 \to I$ so that the graph $(t, r(t, \theta) + \eee, \theta) \subset [-1, 1] \times [-1, 1] \times S^1$ of $r + \epsilon$ is transverse to the graphic and intersects only those regions of the graphic in which $S_s \cap T_{t, \te}$ contains circles that are essential in $T_{t, \te}$.
\end{prop} 

\begin{proof} The only additional case to consider is that in which $(t, r(t, \theta), \theta)$ is a (codimension $3$) vertex in the graphic.  But in that case, the region just below the vertex can be connected to the region just above the vertex by passing through three saddle ``walls".  Thus Lemma \ref{lemma:planar} is still applicable, now with $n = 3$.
\end{proof}

\end{document}